\documentclass[11pt,twoside]{article}
\usepackage{amsmath,amssymb,amsthm,bm,mathrsfs}
\usepackage{graphicx}
\usepackage{epstopdf}
\usepackage{caption}
\usepackage{tabu}
\usepackage{multirow}
\usepackage{subfigure}
\usepackage{booktabs}
\usepackage{diagbox}
\usepackage{listings}
\usepackage[dvipsnames]{xcolor}

\allowdisplaybreaks[4]

\makeatletter
\@addtoreset{equation}{section}
\makeatother

\newtheorem{lm}{Lemma}[section]
\newtheorem{thm}{Theorem}[section]

\newtheorem{cor}{Corollary}[section]
\newtheorem{exam}{Example}[section]

\newcounter{saveeqn}%

\makeatletter
\evensidemargin
\oddsidemargin
\makeatother

\lstdefinestyle{Maple}{
    basicstyle=\small\ttfamily,
    numberstyle=\ttfamily,
    stringstyle=\ttfamily,
    numberstyle=\tiny,
    breaklines=true,
    columns=fixed,
    captionpos=b,
    basewidth=0.47em,
    showstringspaces=false,
    numbers=left,
    frame=utb,
    framextopmargin=2pt,
    framexbottommargin=2pt,
}

\title{\Large\bf{Normal forms of piecewise-smooth systems with a monodromic singular point
}
\thanks{
Supported by National Key R\&D Program of China (No. 2022YFA1005900), NSFC (No. 12271378, No. 12171336),
Sichuan Science and Technology Program (No. 2024NSFJQ0008).
}
}

\author{Jiahao Li,~~Xingwu Chen\!\!
\footnote{Corresponding author. Xingwu Chen (scuxchen@scu.edu.cn, xingwu.chen@hotmail.com).}
,~~Weinian Zhang
\\
{\small School of Mathematics, Sichuan University,}
{\small Chengdu, Sichuan 610064, P. R. China}
}

\date{}

\begin{document}
\maketitle

\begin{abstract}
Normal form theory is developed deeply for planar smooth systems but has few results for piecewise-smooth systems because difficulties arise from continuity of the near-identity transformation, which is constructed piecewise.
%Continuing the second-order normal form for FF equilibrium given in [Disc. Cont. Dyn. Syst. {\bf 36}(2016) 6715-6736],
In this paper, we overcome the difficulties to study normal forms for piecewise-smooth systems with FF, FP, or PP equilibrium and obtain explicit any-order normal forms by finding piecewise-analytic homeomorphisms and deriving a new normal form for analytic systems. Our theorems of normal forms not only generalize previous results from second-order to any-order, from FF type to all FF, FP, PP types, but also provide a new method to compute Lyapunov constants, which are applied to solve the center problem and any-order Hopf bifurcations of piecewise-smooth systems.

\end{abstract}

\baselineskip 15pt%设置行间距
\parskip 10pt%设置段间距
\thispagestyle{empty}%当前页页码不显示
\setcounter{page}{1}%手动设置页码的编号,从1开始计数

%%%%%%%%%%%%%%%%%%%%%%%%%%%%%%%%%%%%%%%%%%%%%%%%%%%%%%%%%%%%%%%%%%%%%%%%
\section{Introduction and main results}

\setcounter{equation}{0}
\setcounter{lm}{0}
\setcounter{thm}{0}
\setcounter{rmk}{0}
\setcounter{df}{0}
\setcounter{cor}{0}
%从0开始计数

Normal form theory is a powerful tool to simplify differential equations and dynamical systems,
which has been widely used in Hilbert's sixteenth problem, bifurcation theory and so on.
Mature theory of normal forms for classical smooth differential systems can be found in textbooks such as \cite{AR,CHO,DUM,GUC,KUZ,Li}.
In recent decades, people develop the normal form theory to systems with lower smoothness (even discontinuity), infinite dimension, random or  full-null degeneracy.
For instance, the idea of normal form theory is used to analyze non-smooth systems in \cite{BUZC,BUZ,WE}, functional differential systems in
\cite{FA,FAR,GUO}, systems with fully null linear part in \cite{GUOZ, RUA} and random dynamical systems in \cite{LIW,LU}.
In this paper we focus on the normal form theory of piecewise-smooth systems, which have received much attention because of
 plenty practical applications in electronic systems \cite{BEBU,SCH}, mechanical systems with dry friction \cite{CH,CHE}, neural activities \cite{COO,TA} and so on.

Consider a piecewise-smooth system
\begin{align}\label{ge}
		\left( \begin{array}{c}
		\dot{x}\\
		\dot{y}\\
	\end{array} \right) =\left\{
    \begin{aligned}
&		\left(X^+(x,y),Y^+(x,y)\right)^\top ~~~~&&\mathrm{if}~y>0,\\
&		\left(X^-(x,y),Y^-(x,y)\right)^\top   && \mathrm{if}~y<0,
	\end{aligned} \right.
\end{align}
where $X^{\pm}(x, y), Y^{\pm}(x, y)$ are real analytic functions.
This system is split by the line $y=0$, called a {\it switching line},
into two subsystems.
As in textbooks \cite{BDM,FAF,JMR,KM}, the subsystem defined on the half plane $y>0$ (resp, $y<0$) is called the {\it upper} (resp. {\it lower}) subsystem and orbits of system~\eqref{ge} are defined piecewise for $y > 0, = 0, < 0$.
Suppose that system~\eqref{ge} has a monodromic singular point $O:(0,0)$, i.e., no orbits approaching $O$ in a definite direction as $t\to \pm\infty$ (see \cite{AN,ROM,ZZ}).
Different from smooth differential systems,
 in piecewise-smooth system~\eqref{ge} there are three types of monodromic singular points:
{\it focus-focus} (FF), {\it focus-parabolic} (FP) and {\it parabolic-parabolic} (PP),
as given in \cite{CO}.
For such three types, usually $O$ is also regarded as a focus of \eqref{ge} and the readers can see
\cite[FIG.1]{CO} or Figure~\ref{MSP} directly for convenience.
Clearly, $O$ is either a focus (see \cite{FAF,HAN}) or an invisible tangent point
(visibility of tangent point can be found in \cite{JEF,KUZ1}) for each subsystem.
Note that PF type can be reduced to FP type by the transformation $(x,y,t)\to (-x,-y,t)$.

\begin{figure}[htp]
  \centering
  \includegraphics[width=15cm]{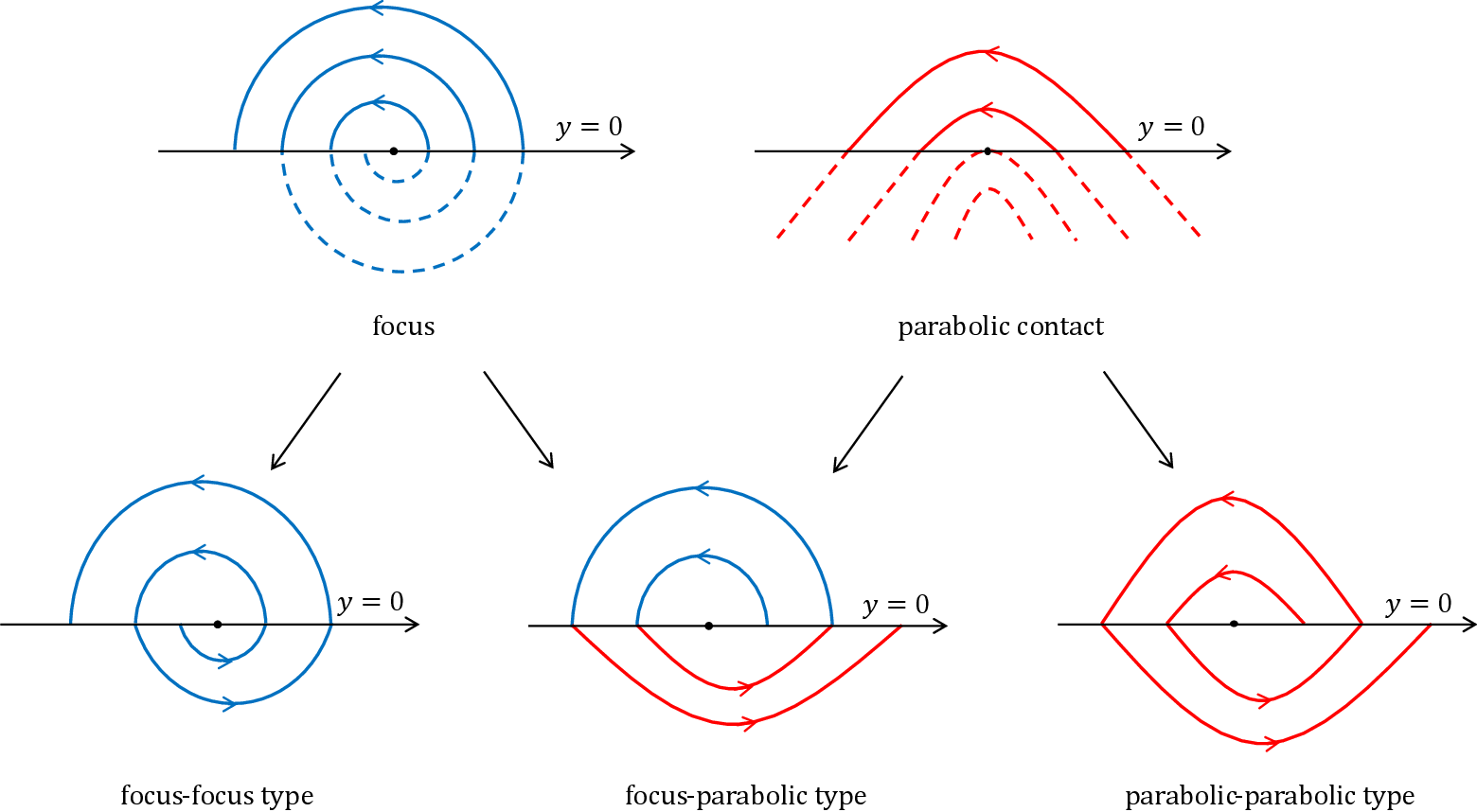}
  \caption{\footnotesize Three types of monodromic singular points of system~\eqref{ge}. }\label{MSP}
\end{figure}

\iffalse
As in the qualitative theory of smooth planar vector fields,
center-focus and cyclicity problems are two very interesting and important challenges for piecewise-smooth system~\eqref{ge}.
That is, to determine whether the monodromic singular point $O$ is a center or a focus and to estimate the number of limit cycles can bifurcate from $O$ (see more details in \cite{ROM}).
Like for smooth system, those problems can also be studied by finding and analyzing displacement map.
\fi

Without loss of generality, suppose that orbits of system~\eqref{ge} rotate around monodromic singular point $O$
of FF, FP or PP type counterclockwise.
Let $\Pi^+(x)$ (resp. $\Pi^-(x)$) be the $x-$coordinate of the first intersection point between
the orbit passing through $(x,0)$ and the switching line for $x\in (0,\varepsilon^+)$ (resp. $(-\varepsilon^-,0)$),
where $0<\varepsilon^\pm\ll 1$.
Usually, the definition of displacement map $\Delta(x)$ for piecewise-smooth system~\eqref{ge} follows the smooth case via the
composition $\Pi^-\circ\Pi^+(x)$ of $\Pi^{\pm}$ and is defined as $\Pi^-\circ\Pi^+(x)-x$, see \cite{CO} and Figure~\ref{DF}(a).
To avoid the composition, another equivalent way is to define the displacement map $\Delta(x)$ as $(\Pi^-)^{-1}(x)-\Pi^+(x)$ as in
\cite{CX2,CO,HAN} and seen in Figure~\ref{DF}(b), and call $\Pi^+(x)$ and $(\Pi^-)^{-1}(x)$
the {\it upper half return map} and {\it lower half return map} respectively.

\begin{figure}[htp]
\centering{}
\subfigure[ $\Pi^-(\Pi^+(x))-x$]{
	\scalebox{0.65}[0.65]{
	\includegraphics{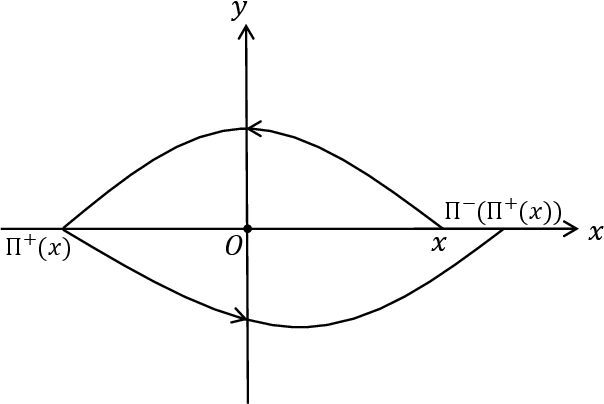}
	}
}
\subfigure[$(\Pi^-)^{-1}(x)-\Pi^+(x)$]{
	\scalebox{0.65}[0.65]{
	\includegraphics{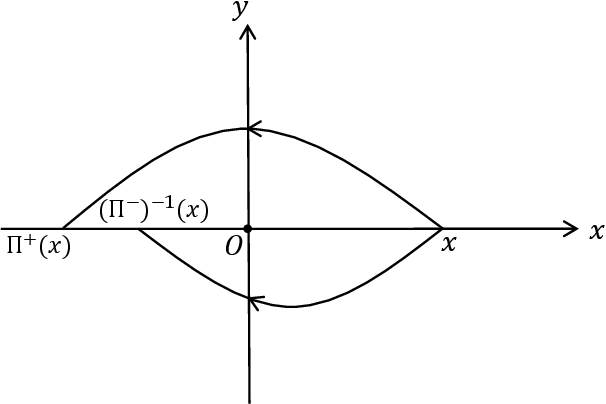}
	}
}
\caption{\footnotesize Displacement map. }
\label{DF}
\end{figure}

By the analyticity of the upper subsystem and the lower one
in \eqref{ge}, we get
\begin{equation}\label{DeltaF}
  \Delta(x)=(\Pi^-)^{-1}(x)-\Pi^+(x)=\sum_{k\ge 1}V_kx^k
\end{equation}
for $x\in (0,\varepsilon^+)$ and coefficient $V_k$ is called the {\it $k$-th Lyapunov constant} usually.
The first nonzero Lyapunov constant decides the local stability of $O$.
Note that the number $k$ of the first nonzero $V_k$ for smooth systems is always odd,
but it can be even for piecewise-smooth system~\eqref{ge}.
In order to uniform the notion of order for both the piecewise-smooth case and smooth one, we introduce the
fractional order as in \cite{CX1}. Precisely, we call the monodromic singular point $O$ of \eqref{ge}
is a focus of order $\varsigma\in \bigcup_{i=0}^{+\infty}\{i/2\}$ if $V_1=\cdots=V_{2\varsigma}=0$ and $V_{2\varsigma+1}\ne0$.
Two particular cases are order $0$ and order $+\infty$, which correspond to rough focus and center respectively.
Clearly, the order of focus in smooth systems is always integer (see \cite{ROM}).

Lyapunov constant $V_k$ is important to distinguish focus of finite order from center (so-called center-focus problem)
and determine the number of limit cycles bifurcating from
focus (Hopf bifurcation problem) for piecewise-smooth system~\eqref{ge}.
Lyapunov constant $V_k$
can be computed by using polar coordinate transformation for focus side and generalized polar coordinate transformation for parabolic side.
Because of tremendous trigonometric polynomials in integration,
there is not a general formula for computation of Lyapunov constants of FF and FP types yet,
but recently
a general formula for computation of Lyapunov constants of PP type was given in \cite{DD}
for system~\eqref{ge} with $(2k^+,2k^-)$-monodromic tangential singularity by considering auxiliary sections which are transversal to both the orbits and the switching line.
However, as said in \cite[p.572]{DD},
it involves some very cumbersome computations when applying the formula that they obtain if $k^{\pm}$ are arbitrary.

Another method is developing the normal form theory for piecewise-smooth systems for computation of Lyapunov constants.
The significant difficulty is to guarantee the continuity of the near-identity transformation (abb. NIT)
defined piecewise. This reminds even for linear normalization that we may not be able to reduce
both the upper subsystem and the lower one to Jordan forms.
For FF type, the first order normal form (i.e., the linear one) and the second one (i.e., the simplest nonlinear one)
are obtained in \cite{CX} by piecewise-linear transformations
and piecewise near-identity transformations respectively, and hence the first two Lyapunov constants are given.
For PP type, a normal form is given in \cite{EM} but, Lyapunov constants are unknown except in \cite{EM1}
for $(2,2)$-monodromic tangential singularity, a special case of PP type.
So far, no results on normal forms of FP type are found.

In this paper, we obtain general normal forms for FF, FP and PP types and Lyapunov constants
for piecewise-smooth system~\eqref{ge}. Our main results are the following theorems.
\begin{thm}
For any positive integer $N$, system~\eqref{ge} has FF normal form
  \begin{align}\label{NFF}
		\left( \begin{array}{c}
		\dot{x}\\
		\dot{y}\\
	\end{array} \right) =\left\{
    \begin{aligned}
&		\left(\begin{array}{cc}
       \gamma_1^+ & -1 \\
       1 & \gamma_1^+ \\
     \end{array}\right)
     \left(\begin{array}{c}
         x \\
         y \\
       \end{array}\right)+
       \sum_{i=1}^{N}\gamma_{i+1}^+y^i\left(\begin{array}{c}
         x \\
         y \\
       \end{array}\right)+{\boldsymbol G}^+_{N+2}(x,y)~~~~&&\mathrm{if}~y>0,\\
&		\left(\begin{array}{cc}
       \gamma_1^- & -1 \\
       1 & \gamma_1^- \\
     \end{array}\right)
     \left(\begin{array}{c}
         x \\
         y \\
       \end{array}\right)+
       \sum_{i=1}^{N}\gamma_{i+1}^-y^i\left(\begin{array}{c}
         x \\
         y \\
       \end{array}\right)+{\boldsymbol G}^-_{N+2}(x,y)   && \mathrm{if}~y<0
	\end{aligned} \right.
\end{align}
if the origin $O$ is a monodromic singular point of FF type,
FP normal form
      \begin{align}\label{NFP}
		\left( \begin{array}{c}
		\dot{x}\\
		\dot{y}\\
	\end{array} \right) =\left\{
    \begin{aligned}
&		\left(\begin{array}{cc}
       \gamma_1^+ & -1 \\
       1 & \gamma_1^+ \\
     \end{array}\right)
     \left(\begin{array}{c}
         x \\
         y \\
       \end{array}\right)+
       \sum_{i=1}^{N}\gamma_{i+1}^+y^i\left(\begin{array}{c}
         x \\
         y \\
       \end{array}\right)+{\boldsymbol G}^+_{N+2}(x,y)~~~~&&\mathrm{if}~y>0,\\
&	   \left(\begin{array}{c}
         1 \\
         x^{2\ell^--1}\\
       \end{array}\right)+
       {\boldsymbol R}^-_{N}(x,y)  && \mathrm{if}~y<0
	\end{aligned} \right.
\end{align}
if the origin $O$ is a monodromic singular point of FP type, and
PP normal form
      \begin{align}\label{NPP}
		\left( \begin{array}{c}
		\dot{x}\\
		\dot{y}\\
	\end{array} \right) =\left\{
    \begin{aligned}
&		\left(\begin{array}{c}
         -1 \\
         x^{2\ell^+-1}+\sum\limits_{i=1}^{N}\sigma_{i+1}^+x^{2\ell^+-1+i}\\
       \end{array}\right)+
      {\boldsymbol R}^+_{N}(x,y)
        && \mathrm{if}~y>0,\\
&		      \left(\begin{array}{c}
         1 \\
         x^{2\ell^--1}\\
       \end{array}\right)+
       {\boldsymbol R}^-_{N}(x,y) && \mathrm{if}~y<0
	\end{aligned} \right.
\end{align}
if the origin $O$ is a monodromic singular point of PP type,
where $\ell^\pm\ge 1$, ${\boldsymbol G}^{\pm}_{N+2}(x,y)$ are sums of homogeneous  polynomials of degree greater or equal to $N+2$
and ${\boldsymbol R}^{\pm}_{N}(x,y)$ are sums of quasi-homogeneous polynomials of type $(1,2\ell^{\pm})$ with
degree greater or equal to $N$.
\label{NFFFPP}
\end{thm}

Note that as defined in \cite{AL}, a scalar polynomial function $f(x,y)$ and a polynomial vector field $\boldsymbol{f}(x,y)$ are called quasi-homogeneous of type $(p, q)$ and degree $k$ if for any $\lambda\in \mathbb{R}_+$ they respectively satisfy $f(\lambda^px,\lambda^qy)=\lambda^kf(x,y)$ and $$\boldsymbol{f}(\lambda^px,\lambda^qy)=
      \left(
      \begin{array}{cc}
      \lambda^{k+p} & 0 \\
      0 & \lambda^{k+q} \\
      \end{array}
      \right)\boldsymbol{f}(x,y).$$

\begin{thm}
For FF normal form~\eqref{NFF}, FP normal form~\eqref{NFP} and PP normal form~\eqref{NPP}, first nonzero Lyapunov constants $V_k$ ($1\le k\le N+1$) of $O$ are given respectively by
\begin{eqnarray*}
\begin{aligned}
&
V_k=\left\{\begin{aligned}
&e^{\gamma_1^+\pi}-e^{-\gamma_1^-\pi} &&{\rm if}~ k=1,\\
  &\left(\gamma_k^+\!+\!(\!-1)^{k-1}\gamma_k^-\right)C_k^{FF}&&{\rm if}~ k\ge 2;
  \end{aligned}\right.
\\
&
 V_k=\left\{\begin{aligned}
&e^{\gamma_1^+\pi}-1&&{\rm if}~k=1,\\
  &\gamma_k^+C_k^{FP}
   ~~~~~~~~~~~~~~~~~~~~&&{\rm if}~ k\ge 2;
  \end{aligned}\right.
\\
&
V_k=
\left\{
\begin{aligned}
&0&&{\rm if}~k~{\rm is~odd},\\
&
\sigma_k^+C_k^{PP}~~~~~~~~~~~~~~~~~~~~ &&{\rm if}~k~{\rm is~even};
\end{aligned}
\right.
\end{aligned}
\end{eqnarray*}
where
\begin{eqnarray*}
&&C_k^{FF}:=\frac{\pi(k-2)!!e^{\gamma_1^+\pi}}{(k-1)!!
   \left(1+(\gamma_1^+)^2\right)^{\frac{k-1}{2}}} \left(\frac{1+e^{\gamma_1^+\pi}}{\pi\sqrt{1+(\gamma_1^+)^2}
   }\right)^{\frac{1+(-1)^k}{2}}
   >0,
\\
&&C_k^{FP}:=\frac{\pi(k-2)!!}{(k-1)!!}
\left(\frac{2}{\pi}\right)^{\frac{1+(-1)^k}{2}}
>0,
\   \
C_k^{PP}:=\frac{2}{k+2\ell^+-1}>0.
\end{eqnarray*}
\label{Lya-FFPP}
\end{thm}

Clearly, by Theorem~\ref{Lya-FFPP} Lyapunov constant $V_k$ is algebraically equivalent
to
$\gamma_k^++(-1)^{k-1}\gamma_k^-$ for FF normal form and to $\gamma_k^+$ for FP normal form when $k\ge 1$,
and to $\mu_k^-$ for PP normal form when $k$ is even.
On the other hand, by Theorem~\ref{Lya-FFPP} the order of weak focus $O$ of system~\eqref{ge} is always non-integer when
$O$ is of PP type because $V_k=0$ for all odd $k$, which is consistent with \cite{DD}.
In order to determine the order of weak focus $O$ of system~\eqref{ge}, which
is important to obtain an upper bound of crossing limit cycles via Hopf bifurcations,
we give further information of the order by the subsystems as follows.

\begin{thm}
Suppose that the origin $O$ is a weak focus of order $\varsigma^+\in \mathbb{N}\cup\{+\infty \}$ of the upper
subsystem of \eqref{ge}. When $O$ is a weak focus of order $\varsigma^-\in \mathbb{N}\cup\{+\infty \}$ of the lower
subsystem, then the order $\varsigma$ of monodromic singular point $O$ of FF type of system~\eqref{ge} satisfies
\begin{description}
\item[(a)] $\varsigma=0$ if $\varsigma^+=0$(resp. $\varsigma^-=0$) and $\varsigma^-$(resp. $\varsigma^+$) $\in \mathbb{Z}^+\cup\{+\infty \}$ ;
\item[(b)] $\varsigma \in \left(\bigcup\limits_{i=0}^{m-1}\left\{\frac{2i+1}{2}\right\}\right)\cup \{m\}$ if $\varsigma^{\pm} \in \mathbb{Z}^+\cup\{+\infty \}$ and $\varsigma^+\ne \varsigma^-$, where $m:=\min\{\varsigma^+,\varsigma^-\}$;
\item[(c)] $\varsigma \in \left(\bigcup\limits_{i=0}^{\varsigma^+-1}\left\{\frac{2i+1}{2}\right\}\right)\cup \left(\bigcup\limits_{i=2\varsigma^+}^{+\infty}\{\frac{i}{2}\}\right)$ if $\varsigma^{\pm} \in \mathbb{Z}^+\cup\{+\infty \}$ and $\varsigma^+= \varsigma^-$.
\end{description}
When $O$ is an invisible tangent point of the lower
subsystem,
then the order $\varsigma$ of monodromic singular point $O$ of FP type of system~\eqref{ge} satisfies
\begin{description}
  \item[(d)] $\varsigma=0$ if $\varsigma^+=0$;
  \item[(e)] $\varsigma \in \left(\bigcup\limits_{i=0}^{\varsigma^+-1}\left\{\frac{2i+1}{2}\right\}\right)\cup \{\varsigma^+\}$ if $\varsigma^+\in \mathbb{Z}^+\cup\{+\infty \}$.
\end{description}
\label{OFF}
\end{thm}

From Theorem~\ref{OFF}, we see that the order $\varsigma$ of  weak focus $O$ of system~\eqref{ge}
does not depend on the multiplicity of tangent point $O$ of one subsystem when
$O$ is of FP type. For PP type, there is no clear relationship between the order $\varsigma$ and
the multiplicity of tangent point. There exist examples with $\varsigma$ greater than or less than or equal to
the multiplicity of tangent point. On the other hand,
by the proof of Theorem~\ref{OFF} given in section 3 there are examples satisfying
$\varsigma=s$ for any given $s\in \bigcup_{i=0}^{+\infty}\{i/2\}$ if $\varsigma^+=\varsigma^-=0$.

This paper is organized as follows.
We give some preliminary lemmas about using appropriate homeomorphism which preserves switching line $y=0$ to reduce subsystem of system~\eqref{ge} to normal form and computing responding return map in section 2.
The proofs of Theorems~\ref{NFFFPP}, \ref{Lya-FFPP} and \ref{OFF} are given in section 3.
In section 4 we give some interesting corollaries and discussions with respect to the relationships between our normal forms and former
results about smooth systems and non-smooth one.
In section 5, we apply our main results to the center problem
and to those systems for correction in some published works.

%%%%%%%%%%%%%%%%%%%%%%%%%%%%%%%%%%
%\section{Preliminaries}
%\setcounter{equation}{0}
%\setcounter{lm}{0}
%\setcounter{thm}{0}
%\setcounter{rmk}{0}
%\setcounter{df}{0}
%\setcounter{cor}{0}

%%%%%%%%%%%%%%%%%%%%%%%%%%%%%%%%%%
\section{Preliminary lemmas}
\setcounter{equation}{0}
\setcounter{lm}{0}
\setcounter{thm}{0}
\setcounter{rmk}{0}
\setcounter{df}{0}
\setcounter{cor}{0}

Before proving main results, we give some preliminary lemmas in this section.

\begin{lm}\label{niih}
Transformation
\begin{equation}\label{geni}
  \left(
    \begin{array}{c}
      x \\
      y \\
    \end{array}
  \right)=
  \left\{\begin{aligned}
  & \left(\begin{array}{cc}
      1 & q_2^+ \\
      0 & q_1^+ \\
    \end{array}\right)
  \left(
    \begin{array}{c}
      u \\
      v \\
    \end{array}
  \right)+
  \left(\begin{array}{c}
      \Phi^+(u,v) \\
      \Psi^+(u,v)\\
    \end{array}\right)~~{\rm for~}v\ge0,\\
    &\left(\begin{array}{cc}
      1 & q_2^- \\
      0 & q_1^- \\
    \end{array}\right)
    \left(
    \begin{array}{c}
      u \\
      v \\
    \end{array}
  \right)+
  \left(\begin{array}{c}
       \Phi^-(u,v) \\
      \Psi^-(u,v)\\
    \end{array}\right)~~{\rm for~}v\le0
  \end{aligned}
  \right.
\end{equation}
is a homeomorphism on a small neighborhood $U$ of $O$, where $q_{1}^\pm, q_2^\pm\in \mathbb{R}$, $q_1^+ q_1^->0$ and
$\Phi^{\pm}, \Psi^{\pm}=o(|(u,v)|)$ are real analytic functions satisfying
$\Phi^+(u,0)=\Phi^-(u,0), \Psi^+(u,0)=\Psi^-(u,0)=0$.
\end{lm}

\begin{proof}
We denote the map on the upper (resp. lower) half plane by ${\boldsymbol F}^+(u,v)$ (resp. ${\boldsymbol F}^-(u,v)$).
Since ${\boldsymbol F}^+(u,0)={\boldsymbol F}^-(u,0)$ by $\Phi^+(u,0)=\Phi^-(u,0)$ and $\Psi^+(u,0)=\Psi^-(u,0)=0$,
we get that map~\eqref{geni} is continuous on the $u$-axis and hence continuous on $U$.
From $\Psi^{\pm}(u,0)=0$, we get
$
  y=q_1^\pm v+\Psi^{\pm}(u,v)=q_1^\pm v(1+g^{\pm}(u,v)),
$
where $g^{\pm}(u,v)$ are analytic and satisfy $g^{\pm}(0,0)=0$.
Thus, \eqref{geni} maps $v>0$, $v=0$, $v<0$ to $y>0$, $y=0$, $y<0$ respectively in $U$ if $q_1^+>0$, and to
$y<0$, $y=0$, $y>0$ respectively if $q_1^+<0$.
Moreover, by $\det D{\boldsymbol F}^{\pm}(0,0)=q_1^\pm\ne 0$ and the Inverse Mapping Theorem, \eqref{geni} has inverse map
\begin{equation}\label{ioni}
  \left(
    \begin{array}{c}
      u \\
      v \\
    \end{array}
  \right)=
  \left\{\begin{aligned}
  &{\boldsymbol G}^{+}(x,y),~~y\ge0,\\
  &{\boldsymbol G}^{-}(x,y),~~y\le0
  \end{aligned}
  \right.
\end{equation}
in a small neighbourhood of $O$.
Clearly, \eqref{ioni} maps $y=0$ to $v=0$. Moreover, ${\boldsymbol G}^{+}(x,0)={\boldsymbol G}^{-}(x,0)$ because ${\boldsymbol F}^+(u,0)={\boldsymbol F}^-(u,0)$.
Thus, \eqref{ioni} is continuous. This lemma is proved.
\end{proof}

Since homeomorphism~\eqref{geni} maps $v=0$ to $y=0$, it transforms system~\eqref{ge} to a new system with switching line $v=0$.
Furthermore, \eqref{geni} transforms orbits of upper (resp. lower) subsystem of \eqref{ge} to orbits of upper (resp. lower) subsystem of the new system if $q_1^+>0$, and to orbits of lower (resp. upper) subsystem of the new system if $q_1^+<0$.
Note that homeomorphism~\eqref{geni} is continuous on line $v=0$,
we get that transformation~\eqref{geni} maps orbits of system~\eqref{ge} to orbits of the
new system by the uniqueness of orbits near monodromic singular point $O$ for these two systems.
That is to say, these two systems are topologically equivalent to each other.

Clearly, the change of variables $(x,y)\to (x,-y)$ transforms a lower subsystem into an upper one.
Thus, in the following two lemmas we simplify upper subsystem of system~\eqref{ge} without loss of generality.

\begin{lm}\label{niofs}
Suppose that the upper subsystem of \eqref{ge} has a weak focus $O$ with  eigenvalues $\alpha\pm\beta i$.
Then by a diffeomorphism
\begin{equation}\label{dhfs}
  \left(\begin{array}{c}
      x \\
      y \\
    \end{array}
  \right)\to
   \left(
     \begin{array}{cc}
       1 & q_2 \\
       0 & q_1 \\
     \end{array}
   \right)
  \left(
   \begin{array}{c}
      x \\
      y \\
    \end{array}
  \right)+\sum_{k=2}^{N+1}
  \left(\begin{array}{c}
      \sum\limits_{j=0}^{k}r_{k-j,j}x^{k-j} y^{j} \\
      \sum\limits_{j=1}^{k}s_{k-j,j}x^{k-j} y^{j} \\
    \end{array}\right)
\end{equation}
in a small neighborhood of $O$, the upper subsystem of \eqref{ge} is transformed equivalently into
\begin{equation}\label{nfN}
  \left(\begin{array}{c}
      \dot{x} \\
      \dot{y} \\
    \end{array}\right)
    =\left(\begin{array}{cc}
      \alpha & -\beta \\
      \beta & \alpha \\
    \end{array}\right)
  \left(\begin{array}{c}
      x \\
      y \\
    \end{array}\right)
    +\sum_{k=1}^{N}y^k
    \left(\begin{array}{cc}
      \nu_{k+1} & -\eta_{k+1} \\
      \eta_{k+1} & \nu_{k+1} \\
    \end{array}\right)\left(\begin{array}{c}
      x \\
      y \\
    \end{array}\right)
    +o\left(|(x,y)|^{N+1}\right).
\end{equation}
Moreover, there exists a time scaling
\begin{equation}\label{tsfs}
  dt\to\frac{1}{\beta}\left(1-\sum_{k=1}^{N}T_ky^k\right)dt
\end{equation}
changing system~\eqref{nfN} into
\begin{equation}\label{nffs}
  \left(\begin{array}{c}
      \dot{x} \\
      \dot{y} \\
    \end{array}\right)
    =\left(\begin{array}{cc}
      \gamma_1 & -1 \\
      1 &  \gamma_1 \\
    \end{array}\right)
  \left(\begin{array}{c}
      x \\
      y \\
    \end{array}\right)
    +\sum_{k=1}^{N}{\gamma}_{k+1} y^k
    \left(\begin{array}{c}
      x \\
      y \\
    \end{array}\right)
    +o\left(|(x,y)|^{N+1}\right),
\end{equation}
where $\gamma_1:=\alpha/\beta$ and each $T_k$ is a polynomial in $\eta_2,...,\eta_{k+1}$.
\end{lm}

\begin{proof}
We use variable transformations and time scaling to reduce the upper subsystem of \eqref{ge}. So, this proof is given in following two steps.

{\it Step 1: transform the upper subsystem of \eqref{ge} into \eqref{nfN}.}
By \cite[Theorem 2.2]{CX}, there exists a linear transformation
\begin{equation}\label{tl}
  \left(\begin{array}{c}
      x \\
      y \\
    \end{array}
  \right)\to
   \left(
     \begin{array}{cc}
       1 & q_2 \\
       0 & q_1 \\
     \end{array}
   \right)
   \left(\begin{array}{c}
     x \\
     y \\
   \end{array}
   \right)
\end{equation}
with $q_2>0$ such that the upper subsystem of system~\eqref{ge} is written as
\begin{equation}\label{fcg}
  \left(\begin{array}{c}
      \dot{x} \\
      \dot{y} \\
    \end{array}
  \right)=
  \left(\begin{array}{cc}
      \alpha & -\beta \\
      \beta & \alpha \\
    \end{array}\right)
  \left(\begin{array}{c}
      x \\
      y \\
    \end{array}\right)+
  \left(\begin{array}{c}
      \sum\limits_{i+j\geq 2} {a}_{ij} x^iy^j \\
      \sum\limits_{i+j\geq 2} {b}_{ij} x^iy^j \\
    \end{array}\right).
\end{equation}
For the nonlinear part of system~\eqref{fcg}, we claim that there exists a near-identity transformation
\begin{equation}\label{nifs}
  \left(\begin{array}{c}
      x \\
      y \\
    \end{array}
  \right)\to
  \left(
   \begin{array}{c}
      x \\
      y \\
    \end{array}
  \right)+\sum_{k=2}^{N+1}
  \left(\begin{array}{c}
      \sum\limits_{j=0}^{k}p_{k-j,j}x^{k-j} y^{j} \\
      \sum\limits_{j=1}^{k}q_{k-j,j}x^{k-j} y^{j} \\
    \end{array}\right)
\end{equation}
changing system~\eqref{fcg} into system~\eqref{nfN}.
In fact, let $\Delta^k:=\left\{{\boldsymbol h}^k(x,y)\in H_2^k \mid h^k_2(x,0)=0 \right\}$,
where $H_2^k$ denotes the linear space of all planar vector-valued homogeneous polynomials of degree $k$
and $h^k_2(x,y)$ is the second component of ${\boldsymbol h}^k(x,y)\in H_2^k$.
  Clearly, near-identity transformation
  \begin{equation*}
    \left(\begin{array}{c}
        x \\
        y \\
      \end{array}
    \right)
    \to \left(\begin{array}{c}
        x \\
        y \\
      \end{array}
    \right)+{\boldsymbol g}^k(x,y),~~~~{\boldsymbol g}^k(x,y)\in \Delta^k
  \end{equation*}
  rewrites system~\eqref{fcg} as
  \begin{align}\label{ffk}
    \left(\begin{array}{c}
        \dot{x} \\
        \dot{y} \\
    \end{array}\right)
    =&A \left(\begin{array}{c}
        x \\
        y \\
      \end{array}\right)
      +\left(\begin{array}{c}
      \sum\limits_{i+j= 2}^{k-1} \tilde{a}_{ij} x^iy^j \\\notag
      \sum\limits_{i+j= 2}^{k-1} \tilde{b}_{ij} x^iy^j \\
    \end{array}\right)\\
    &+\left\{\left(\begin{array}{c}
      \sum\limits_{i+j= k} \tilde{a}_{ij} x^iy^j \\
      \sum\limits_{i+j= k} \tilde{b}_{ij} x^iy^j \\
    \end{array}\right)
    -
    \left[ D{\boldsymbol g}^k A
    \left(\begin{array}{c}
        x \\
        y \\
      \end{array}\right)-A{\boldsymbol g}^k(x,y) \right]\right\}+o\left( |(x,y)|^{k}\right),
  \end{align}
  where $A$ is the linear part matrix of system~\eqref{fcg} and $D{\boldsymbol g}^k$ is the Jacobi matrix of ${\boldsymbol g}^k(x,y)$.
  We define a linear operator $ad_A^k:\Delta^k\to H_2^k$ by
  \begin{equation*}
    ad_A^k{\boldsymbol g}^k(x,y)=D{\boldsymbol g}^k A
    \left(\begin{array}{c}
        x \\
        y \\
      \end{array}\right)-A{\boldsymbol g}^k(x,y).
  \end{equation*}
  Let ${\rm Im}~ad_A^k$ be the image of $ad_A^k$ and $\Theta^k$ be a complementary subspace to ${\rm Im}~ad_A^k$ in $H_2^k$, then we get decomposition
$
    H_2^k={\rm Im}~ad_A^k\oplus \Theta^k.
$
  According to the terms of degree $k$ in system~\eqref{ffk}, our claim below system~\eqref{fcg} can be proved
  by finding ${\boldsymbol g}^k(x,y)\in \Delta^k$ and $\nu_{k}, \eta_{k}$ to satisfy
  \begin{equation}\label{ffe}
    \left(\begin{array}{c}
      \sum\limits_{i+j= k} \tilde{a}_{ij} x^iy^j \\
      \sum\limits_{i+j= k} \tilde{b}_{ij} x^iy^j \\
    \end{array}\right)-ad_A^k{\boldsymbol g}^k(x,y)=
    \left(\begin{array}{c}
      \nu_{k}xy^{k-1}-\eta_{k}y^k \\
      \eta_{k}xy^{k-1}+\nu_{k}y^k \\
    \end{array}\right)\in \Theta^k.
  \end{equation}

  Clearly, the set $\left\{(x^2,0)^\top, (xy,0)^\top, (y^2,0)^\top, (0,xy)^\top, (0,y^2)^\top \right\}$ is a basic of $\Delta^2$,
  and straight computation shows that
  \begin{align*}
    &ad_A^2 \left(\begin{array}{c}
        x^2 \\
        0 \\
      \end{array}\right) =
      \left(
        \begin{array}{c}
          \alpha x^2 -2\beta xy \\
          -\beta x^2 \\
        \end{array}
      \right),~~~~
    ad_A^2 \left(\begin{array}{c}
        xy \\
        0 \\
      \end{array}\right) =
      \left(
        \begin{array}{c}
          \beta x^2 + \alpha xy-\beta y^2 \\
          -\beta xy \\
        \end{array}
      \right),\\
    &ad_A^2 \left(\begin{array}{c}
        y^2 \\
        0 \\
      \end{array}\right) =
      \left(
        \begin{array}{c}
          2 \beta xy + \alpha y^2 \\
          -\beta y^2 \\
        \end{array}
      \right),~~~~
    ad_A^2 \left(\begin{array}{c}
        0 \\
        xy \\
      \end{array}\right) =
       \left(
        \begin{array}{c}
          \beta xy \\
          \beta x^2 + \alpha xy-\beta y^2 \\
        \end{array}
      \right),\\
    &ad_A^2 \left(\begin{array}{c}
        0 \\
        y^2 \\
      \end{array}\right)=
      \left(
        \begin{array}{c}
        \beta y^2\\
        2 \beta xy + \alpha y^2 \\
        \end{array}
      \right).
  \end{align*}
  Associated with the linearity of $ad_A^2$ , we get
  \begin{align*}
    ad_A^2{\boldsymbol g}^2(x,y) & \!=\!ad_A^2\left[
    p_{20}\left(\begin{array}{c}
        x^2 \\
        0 \\
      \end{array}\right)\!+\!
    p_{11}\left(\begin{array}{c}
        xy \\
        0 \\
      \end{array}\right)
      \!+\!p_{02}\left(\begin{array}{c}
        y^2 \\
        0 \\
      \end{array}\right)
      \!+\!q_{11}\left(\begin{array}{c}
        0 \\
        xy \\
      \end{array}\right)
      \!+\!q_{02}\left(\begin{array}{c}
        0 \\
        y^2 \\
      \end{array}\right)
    \right] \\
    & =\!p_{20}ad_A^2\left(\!\begin{array}{c}
        x^2 \\
        0 \\
    \end{array}\!\right)\!+\!
    p_{11}ad_A^2\left(\!\begin{array}{c}
        xy \\
        0 \\
      \end{array}\!\right)
      \!+\!p_{02}ad_A^2\left(\!\begin{array}{c}
        y^2 \\
        0 \\
      \end{array}\!\right)
      \!+\!q_{11}ad_A^2\left(\!\begin{array}{c}
        0 \\
        xy \\
      \end{array}\!\right)
      \!+\!q_{02}ad_A^2\left(\!\begin{array}{c}
        0 \\
        y^2 \\
      \end{array}\!\right)\\
      &=\!\left(
          \begin{array}{c}
            c_{20} x^2 +c_{11}xy+c_{02} y^2\\
            d_{20} x^2 +d_{11}xy+d_{02} y^2
          \end{array}
        \right),
  \end{align*}
  where
  \begin{align*}
    c_{20} &:=\alpha p_{20}+\beta p_{11},~~
    c_{11}:=-2\beta p_{20}+\alpha p_{11}+2\beta p_{02}+\beta q_{11} ,~~
    c_{02}:=-\beta p_{11}+\alpha p_{02}+\beta q_{02},\\
    d_{20}&:=\beta q_{11}-\beta p_{20},~~
    d_{11}:=\alpha q_{11}+2\beta q_{02}-\beta p_{11},~~~~~~~~~~~~~~~
    d_{02}:=-\beta q_{11}+\alpha q_{02}-\beta p_{02}.
  \end{align*}
  From \eqref{ffe}, we get the following relation
  \begin{equation*}
    \tilde{a}_{20}-c_{20}=0,~\tilde{a}_{11}-c_{11}=\nu_2,~
    \tilde{a}_{02}-c_{02}=-\eta_2,~
    \tilde{b}_{20}-d_{20}=0,~\tilde{b}_{11}-d_{11}=\eta_2,~
    \tilde{b}_{02}-d_{02}=\nu_2,
  \end{equation*}
  which is an algebraic system in variables $p_{20},p_{11},p_{02},q_{11},q_{02},\nu_2,\eta_2$.
  Solving it, we obtain
  \begin{align*}
    p_{20}&=C,~~~~~
    p_{11}=\frac{\tilde{a}_{20}}{\beta}-\frac{\alpha }{\beta}C,~~~~~
    q_{11}=\frac{\tilde{b}_{20}}{\beta}+C,\\
    p_{02}&=\frac{3\beta^{2}\tilde{a}_{11}-
    (6\beta^{2}+\alpha^2)\tilde{b}_{20}-
    3\beta^{2}\tilde{b}_{02}+
    \alpha\beta \tilde{a}_{02}-
    \alpha\beta \tilde{a}_{20}+
    \alpha\beta\tilde{b}_{11}}{\beta(9\beta^{2}+\alpha^{2})},\\
    q_{02}&=\frac{3\beta^{2}\tilde{a}_{02}+
    (6\beta^{2}+\alpha^2)\tilde{a}_{20}+
    3\beta^{2}\tilde{b}_{11}-
    \alpha\beta\tilde{a}_{11}+
    \alpha\beta\tilde{b}_{02}-
    \alpha\beta\tilde{b}_{20}}{\beta(9\beta^{2}+\alpha^{2})}
    -\frac{\alpha }{\beta}C,\\
    \nu_2&=\frac{\beta(3\beta^{2}+\alpha^{2})
    (\tilde{a}_{11}+\tilde{b}_{20})\!+\!
    6\beta^{3}\tilde{b}_{02}\!-\!
    2\alpha\beta^{2}(\tilde{a}_{02}+\tilde{b}_{11})\!-\!
    \alpha(7\beta^{2}+\alpha^{2})\tilde{a}_{20}}
    {\beta(9\beta^{2}+\alpha^{2})}
    +\frac{\beta^2+\alpha^2 }{\beta}C,\\
    \eta_2&=\frac{\beta(3\beta^{2}+\alpha^{2})
    (\tilde{b}_{11}-\tilde{a}_{20})\!-\!
    6\beta^{3}\tilde{a}_{02}\!+\!
    2\alpha\beta^{2}(\tilde{a}_{11}-\tilde{b}_{02})\!-\!
    \alpha(7\beta^{2}+\alpha^{2})\tilde{b}_{20}}
    {\beta(9\beta^{2}+\alpha^{2})},
  \end{align*}
  where $C$ is an arbitrary constant.

  Assume that there exists a near-identity transformation
\begin{equation}
  \left(\begin{array}{c}
      x \\
      y \\
    \end{array}
  \right)\to\left(
   \begin{array}{c}
      x \\
      y \\
    \end{array}
  \right)+\sum_{k=2}^{m}
  \left(\begin{array}{c}
      \sum\limits_{j=0}^{k}p_{k-j,j}x^{k-j} y^{j} \\
      \sum\limits_{j=1}^{k}q_{k-j,j}x^{k-j} y^{j} \\
    \end{array}\right)
    \label{uhoffm}
\end{equation}
rewriting system~\eqref{fcg} as
\begin{equation*}
\left(\begin{array}{c}
      \dot{x} \\
      \dot{y} \\
    \end{array}\right)=A
  \left(\begin{array}{c}
      x \\
      y \\
    \end{array}\right)
    +\sum_{k=1}^{m-1}y^k
    \left(\begin{array}{cc}
      \nu_{k+1} & -\eta_{k+1} \\
      \eta_{k+1} & \nu_{k+1} \\
    \end{array}\right)\left(\begin{array}{c}
      x \\
      y \\
    \end{array}\right)
    +
    \left(\begin{array}{c}
      \sum\limits_{i+j= m+1} \hat{a}_{ij} x^iy^j \\
      \sum\limits_{i+j= m+1} \hat{b}_{ij} x^iy^j \\
    \end{array}\right)+h.o.t.
\end{equation*}
  in a neighbourhood of the origin.
  It is not hard to verify that $\Delta^{m+1}$ has a basic consisting
  of vectors $(x^{m-k+1}y^k,0)^\top$ for $0\le k\le m+1$ and vectors $(0, x^{m-k+1}y^k)^\top$ for $1\le k\le m+1$,
  and
  \begin{equation}\label{ffm1}
    ad_A^{m+1}
    \left(
      \begin{array}{c}
        x^{m-k+1}y^k \\
        0 \\
      \end{array}
    \right)
    =
    \left(
      \begin{array}{c}
       \Gamma_1^k\\
        -\Gamma_2^k\\
      \end{array}
    \right),~~~
     ad_A^{m+1}
    \left(
      \begin{array}{c}
        0 \\
        x^{m-k+1}y^k \\
      \end{array}
    \right)  =
    \left(
      \begin{array}{c}
      \Gamma_2^k \\
      \Gamma_1^k \\
      \end{array}
    \right),
  \end{equation}
  where
  $\Gamma_1^k:=k\beta x^{m-k+2}y^{k-1}\!+\!m\alpha x^{m-k+1}y^k\!-\!(m\!-\!k\!+\!1)\beta x^{m-k}y^{k+1},
  \Gamma_2^k:=\beta x^{m-k+1} y^k.$
  By the linearity of $ad_A^{m+1}$ and \eqref{ffm1}, we get
$$
  \begin{aligned}
&    ad_A^{m+1}\left[\sum_{k=0}^{m+1}p_{m-k+1,k}
    \left(
      \begin{array}{c}
        x^{m-k+1}y^k \\
        0 \\
      \end{array}
    \right)+\sum_{k=1}^{m+1}q_{m-k+1,k}
    \left(
      \begin{array}{c}
        0 \\
        x^{m-k+1}y^k \\
      \end{array}
    \right)\right]
    \\
    =&\sum_{k=0}^{m+1}p_{m-k+1,k}
    ad_A^{m+1}\left(
      \begin{array}{c}
        x^{m-k+1}y^k \\
        0 \\
      \end{array}
    \right)+\sum_{k=1}^{m+1}q_{m-k+1,k}
    ad_A^{m+1}\left(
      \begin{array}{c}
        0 \\
        x^{m-k+1}y^k \\
      \end{array}
    \right)\\
    =&\sum_{k=0}^{m+1}p_{m-k+1,k}
    \left(
      \begin{array}{c}
       \Gamma_1^k\\
        -\Gamma_2^k\\
      \end{array}
    \right)+\sum_{k=1}^{m+1}q_{m-k+1,k}
    \left(
      \begin{array}{c}
       \Gamma_2^k\\
        \Gamma_1^k\\
      \end{array}
    \right).
    \end{aligned}
$$
  Thus, the first component $I^1$ of the above vector satisfies
  $$I^1=\sum_{k=0}^{m+1}p_{m-k+1,k}\Gamma_1^k+\sum_{k=1}^{m+1}
  q_{m-k+1,k}\Gamma_2^k=\sum_{k=0}^{m+1}c_{m-k+1,k}x^{m-k+1}y^k,$$
  where
$c_{m+1,0}:=m\alpha p_{m+1,0}+\beta p_{m,1}$, $c_{m,1}:= 2\beta p_{m-1,2}+\left[m\alpha-(m+1)\beta \right] p_{m,1}+\beta q_{m,1}$,
$c_{0,m+1}:=m\alpha p_{0,m+1}+\beta q_{0,m+1} -\beta p_{1,m}$ and for $k=2,...,m$
$$ c_{m-k+1,k}:=(k\!+\!1)\beta p_{m-k,k+1}\!+\!
  m\alpha p_{m-k+1,k}+\beta q_{m-k+1,k}\!-\!(m\!-\!k\!+\!2)\beta p_{m-k+2,k-1}.$$
Similarly, we compute the second component $I^2$ and get
$$I^2:= \sum_{k=0}^{m+1}d_{m-k+1,k}x^{m-k+1}y^k,$$
  where $d_{m+1,0}:=-\beta q_{m+1,0}+\beta q_{m,1}$, $ d_{m,1}:= 2\beta q_{m-1,2}+m\alpha q_{m,1}-\beta p_{m,1}$,
  $d_{0,m+1}:=m\alpha q_{0,m+1}-\beta p_{0,m+1} -\beta q_{1,m}$ and for $k=2,...,m$
  $$d_{m-k+1,k}:=(k\!+\!1)\beta q_{m-k,k+1}\!+\!
  m\alpha q_{m-k+1,k}\!-\!\beta p_{m-k+1,k}\!-\!(m\!-\!k\!+\!2)\beta q_{m-k+2,k-1}.$$

  According to \eqref{ffe}, we need to solve algebraic system
  \begin{equation}\label{ffeqg}
  \left\{\begin{aligned}
   &eq_1:&~&\hat{a}_{m-k+1,k}-c_{m-k+1,k}=0,~~k=0,1,...m-1,\\
   &eq_2:&~&\hat{b}_{m-k+1,k}-d_{m-k+1,k}=0,~~k=0,1,...m-1,\\
   &eq_3:&~&\hat{a}_{1,m}-c_{1,m}=\nu_{m+1},\\
   &eq_4:&~&\hat{b}_{1,m}-d_{1,m} =\eta_{m+1},\\
   &eq_5:&~&\hat{a}_{0,m+1}-c_{0,m+1} =-\eta_{m+1},\\
   &eq_6:&~&\hat{b}_{0,m+1}-d_{0,m+1}=\nu_{m+1}\\
   \end{aligned}\right.
  \end{equation}
 in variables $\nu_{m+1},\eta_{m+1},p_{m-k+1,k},q_{m-k+1,k},k=1,...,m+1$.
  We replace $eq_3,eq_4$ in \eqref{ffeqg} by $eq'_3:=eq_3-eq_6$, $eq'_4:=eq_4+eq_5$ and
  get an equivalent algebraic system $\{eq_1,eq_2,eq'_3,eq'_4,eq_5,eq_6\}$.
  Solving algebraic system $\{eq_1,eq_2\}$, we get
  \begin{equation}\label{pqm1}
    \begin{aligned}
    p_{m+1,0}&=C,~~~~~~~~
    p_{m,1}=\frac{\hat{a}_{m+1,0}}{\beta}-\frac{m\alpha}{ \beta}C,~~~~~~~~~
    q_{m,1}=\frac{\hat{b}_{m+1,0}}{\beta}+C,\\
     p_{m-k,k+1}&=\frac{\hat{a}_{m-k+1,k}+(m\!-\!k\!+\!2)\beta p_{m-k+2,k-1}-m\alpha p_{m-k+1,k}-\beta q_{m-k+1,k}}{(k+1)\beta},\\
    q_{m-k,k+1}&=\frac{\hat{b}_{m-k+1,k}+{\rm sgn}(k\!-\!1)(m\!-\!k\!+\!2)\beta q_{m-k+2,k-1}-m\alpha q_{m-k+1,k}+\beta p_{m-k+1,k}}{(k+1)\beta},
    \end{aligned}
  \end{equation}
  where $k=1,...,m-1$ and $C$ is an arbitrary constant.
  Substituting \eqref{pqm1} in algebraic system $\{eq'_3,eq'_4\}$, we obtain
$$
    (m+2)\beta p_{0,m+1}-m\alpha q_{0,m+1}=\Upsilon_1, ~~~~
    m\alpha p_{0,m+1}+(m+2)\beta q_{0,m+1}=\Upsilon_2,
$$
  where $\Upsilon_1:=\hat{a}_{1,m}-\hat{b}_{0,m+1}+2\beta \left(p_{2,m-1}-q_{1,m}\right)-m\alpha p_{1,m},
  \Upsilon_2:=\hat{a}_{0,m+1}+\hat{b}_{1,m}+2\beta \left(p_{1,m}+q_{2,m-1}\right)-m\alpha q_{1,m}.$
 % \begin{align*}
  %\Upsilon_1&:=\hat{a}_{1,m}-\hat{b}_{0,m+1}+2\beta \left(p_{2,m-1}-q_{1,m}\right)-m\alpha p_{1,m},\\
  %\Upsilon_2&:=\hat{a}_{0,m+1}+\hat{b}_{1,m}+2\beta \left(p_{1,m}+q_{2,m-1}\right)-m\alpha q_{1,m}.
  %\end{align*}
  Then,
  \begin{equation}\label{pqm}
    p_{0,m+1}=\frac{(m+2)\beta \Upsilon_1+m\alpha \Upsilon_2}{m^2\alpha^2+(m+2)^2\beta^2},~~~~
    q_{0,m+1}=\frac{(m+2)\beta \Upsilon_2-m\alpha \Upsilon_1}{m^2\alpha^2+(m+2)^2\beta^2}.
  \end{equation}
Finally,
substituting
\eqref{pqm1} and \eqref{pqm} in
the algebraic system $\{eq_5,eq_6\}$, we obtain
$$
  \begin{aligned}
    \nu_{m+1} & = -m\alpha q_{0,m+1}+\beta p_{0,m+1} +\beta q_{1,m}+\hat{b}_{0,m+1},\\
    \eta_{m+1} & =m\alpha p_{0,m+1}+\beta q_{0,m+1} -\beta p_{1,m}-\hat{a}_{0,m+1},
  \end{aligned}
$$
which shows that the algebraic system $\{eq_1,eq_2,eq_3,eq_4,eq_5,eq_6\}$ is solvable.
The solvability shows from \eqref{ffe} that
we can find a transformation
\begin{equation}\label{nim1}
  \left(\begin{array}{c}
      x \\
      y \\
    \end{array}
  \right)\to
   \left(
   \begin{array}{c}
      x \\
      y \\
    \end{array}
  \right)+\sum_{k=2}^{m+1}
  \left(\begin{array}{c}
      \sum\limits_{j=0}^{k}p_{k-j,j}x^{k-j} y^{j} \\
      \sum\limits_{j=1}^{k}q_{k-j,j}x^{k-j} y^{j} \\
    \end{array}\right)+o\left(|(x,y)|^{m+1}\right)
\end{equation}
to change system~\eqref{fcg} into the form
\begin{equation}\label{nfm1}
  \left(\begin{array}{c}
      \dot{x} \\
      \dot{y} \\
    \end{array}\right)
    =\left(\begin{array}{cc}
      \alpha & -\beta \\
      \beta & \alpha \\
    \end{array}\right)
  \left(\begin{array}{c}
      x \\
      y \\
    \end{array}\right)
    +\sum_{k=1}^{m}y^k
    \left(\begin{array}{cc}
      \nu_{k+1} & -\eta_{k+1} \\
      \eta_{k+1} & \nu_{k+1} \\
    \end{array}\right)\left(\begin{array}{c}
      x \\
      y \\
    \end{array}\right)
    +o\left(|(x,y)|^{m+1}\right)
\end{equation}
in a neighbourhood of $O$.
Note that
the terms in $o\left(|(x,y)|^{m+1}\right)$ in \eqref{nim1} do not
influence the terms of degree no more than $m+1$ in \eqref{nfm1}.
It means that one can get \eqref{nfm1} from \eqref{fcg} by the $(m+1)$-th order truncation of \eqref{nim1}.
Therefore, our claim below system~\eqref{fcg} is proved by the mathematical induction.

Combining the linear transformation~\eqref{tl} with the near-identity transformation~\eqref{nifs},
we find a diffeomorphism~\eqref{dhfs} with $r_{k,0}=p_{k,0}$
$r_{k-j,j}=p_{k-j,j}+q_2q_{k-j,j}$ and $s_{k-j,j}=q_1q_{k-j,j}$, $j=1,...,k$,
and \eqref{dhfs} transforms the upper subsystem of \eqref{ge} into system~\eqref{nfN}.
This completes the proof in Step 1.

{\it Step 2: transform \eqref{nfN} into \eqref{nffs}.}
The time scaling \eqref{tsfs} transforms system~\eqref{nfN} as
\begin{equation*}
  \left(\begin{array}{c}
      \dot{x} \\
      \dot{y} \\
    \end{array}\right)
    =\left(\begin{array}{cc}
      \gamma_1 & -1 \\
      1 & \gamma_1 \\
    \end{array}\right)
  \left(\begin{array}{c}
      x \\
      y \\
    \end{array}\right)
    +\sum_{k=1}^{N}y^k
    \left(\begin{array}{cc}
      \gamma_{k+1} & -\tilde{\eta}_{k+1} \\
      \tilde{\eta}_{k+1} &  \gamma_{k+1} \\
    \end{array}\right)\left(\begin{array}{c}
      x \\
      y \\
    \end{array}\right)
    +h.o.t.,
\end{equation*}
 where $\gamma_1:=\alpha/\beta$, $ \gamma_2:=(\nu_2-\alpha T_1)/\beta$, $\tilde{\eta}_2:=\eta_2/\beta-T_1$ and
 \begin{equation*}
    \gamma_{k+1}:=\frac{1}{\beta}\left(\nu_{k+1}-
    \sum\limits_{i=1}^{k-1}T_{k-i}\nu_{i+1}-\alpha T_k\right),~
   \tilde{\eta}_{k+1}:=\frac{1}{\beta}\left(\eta_{k+1} -\sum\limits_{i=1}^{k-1}T_{k-i}\eta_{i+1}\right)-T_k, ~k=2,...,N.
 \end{equation*}
Solving
the algebraic system $\left\{\tilde{\eta}_{2}=0,...,\tilde{\eta}_{N+1}=0\right\}$, we recursively obtain
$$T_1=\frac{\eta_2}{\beta},~~~~T_k=\frac{1}{\beta}\left(\eta_{k+1} -\sum_{i=1}^{k-1}T_{k-i}\eta_{i+1}\right), k=2,...,N.$$
 Thus \eqref{tsfs} transforms \eqref{nfN} into \eqref{nffs}.
This completes the proof in Step 2 and hence this lemma is proved.
\end{proof}

In order to simplify the upper subsystem of \eqref{ge} with an invisible tangent point $O$, we expand $\Phi^+(x,y)$
and $\Psi^+(x,y)$ in  \eqref{geni} as sums of quasi-homogeneous polynomials as in \cite{EM} and obtain following lemma.

\begin{lm}\label{nfops}
Suppose that the upper subsystem of \eqref{ge} has an invisible tangent point $O$ with multiplicity $2\ell-1$.
Then by a diffeomorphism
\begin{equation}\label{dfps}
  \left(\begin{array}{c}
      x \\
      y \\
    \end{array}
  \right)\to
  \left(
    \begin{array}{cc}
      1 & 0 \\
      0 & q_1 \\
    \end{array}
  \right)
  \left(
   \begin{array}{c}
      x \\
      y \\
    \end{array}
  \right)
  +\sum_{k=1}^{N}
  \left(
    \begin{array}{c}
     \sum\limits_{i=0}^{\lfloor\frac{k+1}{2\ell}\rfloor}r_{k+1-2\ell i,i}x^{k+1-2\ell i}y^i\\
      \sum\limits_{i=1}^{1+\lfloor\frac{k}{2\ell}\rfloor}s_{k+2\ell-2\ell i,i}x^{k+2\ell-2\ell i}y^i \\
    \end{array}
  \right)
\end{equation}
in a small neighborhood of $O$, the upper subsystem of \eqref{ge} is transformed equivalently into
\begin{equation}\label{npN}
  \left(\begin{array}{c}
      \dot{x} \\
      \dot{y} \\
    \end{array}
  \right)=
  \left(\begin{array}{c}
    a_0-\sum\limits_{k=1}^{N}\mu_{k+1} x^k \\
    -a_0x^{2\ell-1}+\sum\limits_{k=1}^{N}\mu_{k+1} x^{k+2\ell-1} \\
  \end{array}\right)+{\boldsymbol R}_{N}(x,y),
\end{equation}
where $a_0:=X^+(0,0)$ and $\lfloor s\rfloor$ means the maximum integer no greater than $s$. Moreover, there exists a time scaling
\begin{equation}\label{tsps}
  dt\to  \frac{1}{-a_0}\left(1-\sum_{k=1}^{N}T_kx^k\right)dt
\end{equation}
changing system~\eqref{npN} into
\begin{equation}\label{nfps}
  \left(\begin{array}{c}
      \dot{x} \\
      \dot{y} \\
    \end{array}
  \right)=
  \left(\begin{array}{c}
    -1 \\
    x^{2\ell-1} \\
  \end{array}\right)+\tilde{\boldsymbol R}_{N}(x,y),
\end{equation}
where each $T_k$ is a polynomial in $\mu_2,...,\mu_{k+1}$.
Here both ${\boldsymbol R}_{N}(x,y)$ and $\tilde{\boldsymbol R}_{N}(x,y)$ are sums of quasi-homogeneous polynomial vector fields
of type $(1,2\ell)$ and degree no less than $N$.
\end{lm}

\begin{proof}
This proof consists of two steps.

{\it Step 1: transform the upper subsystem of \eqref{ge} into \eqref{npN}.}
Since the upper subsystem of \eqref{ge} has an invisible tangent point $O$ with multiplicity $2\ell-1$, we get
\begin{equation*}
  \frac{\partial Y^+(0,0)}{\partial x}=\cdots=\frac{\partial^{2\ell-2}Y^+(0,0)}{\partial x^{2\ell-2}}=0, ~~~~~~
  a_0\frac{\partial^{2\ell-1}Y^+(0,0)}{\partial x^{2\ell-1}}<0.
\end{equation*}
From the analyticity of functions $X^+(x,y)$ and $Y^+(x,y)$ at $O$, the upper subsystem of \eqref{ge} can be expanded as
\begin{equation}\label{Eop}
  \left(
    \begin{array}{c}
      \dot{x} \\
      \dot{y} \\
    \end{array}
  \right)=
  \left(
    \begin{array}{c}
      a_0+\sum\limits_{i+j\ge1}a_{ij}x^iy^j \\
      b_{2\ell-1,0}x^{2\ell-1}+\sum\limits_{i+j\le 2\ell-2}b_{i,j+1}x^iy^{j+1}+
      \sum\limits_{i+j\ge2\ell}b_{ij}x^iy^j \\
    \end{array}
  \right).
\end{equation}
%where
%\begin{equation*}
%  a_{ij}:=\frac{1}{i!j!}\frac{\partial^{i+j}X^+(0,0)}{\partial x^i\partial y^j},~~~~
%  b_{ij}:=\frac{1}{i!j!}\frac{\partial^{i+j}Y^+(0,0)}{\partial x^i\partial y^j}.
%\end{equation*}
Clearly, there exists a linear transformation \eqref{tl}
with $q_2:=0$ and $q_1:=-b_{2\ell-1,0}/{a_0}>0$
such that system~\eqref{Eop} is rewritten as
\begin{equation}\label{pcg}
    \left(
    \begin{array}{c}
      \dot{x} \\
      \dot{y} \\
    \end{array}
  \right)=
  \left(
    \begin{array}{c}
      a_0+\sum\limits_{i+j\ge1}\tilde{a}_{ij}x^iy^j \\
      -a_0 x^{2\ell-1}+\sum\limits_{i+j\le 2\ell-2}\tilde{b}_{i,j+1}x^iy^{j+1}
      +\sum\limits_{i+j\ge2\ell}\tilde{b}_{ij}x^iy^j \\
    \end{array}
  \right),
\end{equation}
where $\tilde{a}_{ij}:=a_{ij}(-b_{2\ell-1,0})^j/a^j_0$ and $\tilde{b}_{ij}:=
b_{ij}(-b_{2\ell-1,0})^{j-1}/a^{j-1}_0$.

We claim that there exists a near-identity transformation
\begin{equation}\label{nips}
  \left(\begin{array}{c}
      x \\
      y \\
    \end{array}
  \right)\to
  \left(
   \begin{array}{c}
      x \\
      y \\
    \end{array}
  \right)+\sum_{k=1}^{N}
  \left(
    \begin{array}{c}
     \sum\limits_{i=0}^{\lfloor\frac{k+1}{2\ell}\rfloor}p_{k+1-2\ell i,i}x^{k+1-2\ell i}y^i\\
      \sum\limits_{i=1}^{1+\lfloor\frac{k}{2\ell}\rfloor}q_{k+2\ell-2\ell i,i}x^{k+2\ell-2\ell i}y^i \\
    \end{array}
  \right)
\end{equation}
changing system~\eqref{pcg} into system~\eqref{npN}.
In fact, let $\tilde{\Delta}^k:=\left\{{\boldsymbol h}^k(x,y)\in \mathscr{P}_2^k\mid h_2^k(x,0)=0 \right\}$, where $\mathscr{P}_2^k$ denotes the linear space of all planar vector-valued quasi-homogeneous polynomials of type $(1,2\ell)$ and degree $k$,
and $h_2^k(x,y)$ is the second component of ${\boldsymbol h}^k(x,y)\in \mathscr{P}_2^k$.
From \cite{EM1},  the right-hand side of system~\eqref{pcg} can be written as $\sum_{k\ge-1}{\boldsymbol F}_k(x,y)$,
where
\begin{equation*}
 {\boldsymbol F}_{-1}(x,y):=
 \left(
   \begin{array}{c}
     a_0 \\
     -a_0x^{2\ell-1}\\
   \end{array}
 \right),~~~
  {\boldsymbol F}_k(x,y):=
  \left(
    \begin{array}{c}
      \sum\limits_{i=0}^{\lfloor\frac{k+1}{2\ell}\rfloor}\tilde{a}_{k+1-2\ell i,i}x^{k+1-2\ell i}y^i\\
      \sum\limits_{i=0}^{1+\lfloor\frac{k}{2\ell}\rfloor}\tilde{b}_{k+2\ell-2\ell i,i}x^{k+2\ell-2\ell i}y^i \\
    \end{array}
  \right),~~k\ge 0.
\end{equation*}
Further, it can be rewritten as
\begin{equation*}
  \sum_{i=-1}^{k-1}{\boldsymbol F}_{i}(x,y)+\left[{\boldsymbol F}_k(x,y)-D{\boldsymbol g}^{k+1}{\boldsymbol F}_{-1}(x,y)+D{\boldsymbol F}_{-1}{\boldsymbol g}^{k+1}(x,y)\right]+h.o.t.
\end{equation*}
by a near-identity transformation
\begin{equation*}
  \left(
    \begin{array}{c}
      x \\
      y \\
    \end{array}
  \right)\to
  \left(
    \begin{array}{c}
      x \\
      y \\
    \end{array}
  \right)+{\boldsymbol g}^{k+1}(x,y),~~~~{\boldsymbol g}^{k+1}(x,y)\in \tilde{\Delta}^{k+1}.
\end{equation*}
Similarly to the proof of Lemma~\ref{niofs}, we define a linear operator $ad_p^k:\tilde{\Delta}^{k+1}\to \mathscr{P}_2^k$ by
\begin{equation*}
  ad_p^k{\boldsymbol g}^{k+1}(x,y):=D{\boldsymbol g}^{k+1}{\boldsymbol F}_{-1}(x,y)-D{\boldsymbol F}_{-1}{\boldsymbol g}^{k+1}(x,y).
\end{equation*}
From decomposition
$
  \mathscr{P}_2^k={\rm Im}~ad_p^k \oplus \Theta^k,
$
where $\Theta^k$ denotes a complementary subspace to ${\rm Im}~ ad_p^k$,
we need only to find ${\boldsymbol g}^{k+1}(x,y) \in \tilde{\Delta}^{k+1}$ and $\mu_{k+2}$ satisfying
\begin{equation}\label{ppe}
  {\boldsymbol F}_k(x,y)-ad_p^k{\boldsymbol g}^{k+1}(x,y)=
  \left(
    \begin{array}{c}
      -\mu_{k+2} x^{k+1} \\
      \mu_{k+2} x^{k+2\ell} \\
    \end{array}
  \right)\in \Theta^k,~~k=0,...,N-1.
\end{equation}

Clearly, $\tilde{\Delta}^1=span\{(x^2,0)^\top, (0,xy)^\top, \omega(y,0)^\top\}$, where $\omega:=\lfloor1/\ell \rfloor$.
It is not hard to compute
\begin{equation*}
  ad_p^0
  \left(\!
    \begin{array}{c}
      x^2 \\
      0 \\
    \end{array}
  \!\right)\!=\!
  \left(\!
    \begin{array}{c}
      2a_0x \\
      (2\ell-1)a_0x^{2\ell} \\
    \end{array}
  \!\right),~
  ad_p^0
  \left(\!
    \begin{array}{c}
      y \\
      0 \\
    \end{array}
  \!\right)\!=\!
  \left(\!
    \begin{array}{c}
      -a_0x \\
      a_0y \\
    \end{array}
  \!\right),~
  ad_p^0
  \left(\!
    \begin{array}{c}
      0 \\
      xy \\
    \end{array}
  \!\right)\!=\!
  \left(\!
    \begin{array}{c}
      0 \\
      -a_0x^{2\ell}+a_0y \\
    \end{array}
  \!\right).
\end{equation*}
According to the linearity of $ad_p^0$, we get
\begin{align*}
  ad_p^0{\boldsymbol g}^1&=p_{20}ad_p^0
  \left(
    \begin{array}{c}
       x^2\\
       0\\
    \end{array}
  \right)+\omega p_{01} ad_p^0
  \left(
    \begin{array}{c}
       y\\
       0\\
    \end{array}
  \right)
  +q_{11}ad_p^0
  \left(
    \begin{array}{c}
      0 \\
      xy \\
    \end{array}
  \right)\\
  &=\left(
    \begin{array}{c}
      a_0(2p_{20}-\omega p_{01})x \\
      \left[(2\ell-1)a_0p_{20}-a_0q_{11}\right]x^{2\ell}+a_0(q_{11}+\omega p_{01})y \\
    \end{array}
  \right).
\end{align*}
By \eqref{ppe}, we get an algebraic system
\begin{equation*}
  \tilde{a}_{10}-2a_0p_{20}+\omega a_0p_{01}=-\mu_2,~~
  \tilde{b}_{2\ell,0}- (2\ell-1)a_0p_{20}+a_0q_{11}=\mu_2,~~
  \tilde{b}_{01}-a_0q_{11}-\omega a_0 p_{01}=0
\end{equation*}
in variables $p_{01},q_{11},p_{20},\mu_2$, which has a solution
\begin{equation*}
  p_{01}=0,~~~~
  q_{11}=\frac{\tilde{b}_{01}}{a_0},~~~~
  p_{20}=\frac{\tilde{a}_{10}+\tilde{b}_{2\ell,0}+\tilde{b}_{01}}{a_0(2\ell+1)},~~~~
  \mu_2=\frac{2\tilde{b}_{2\ell,0}+2\tilde{b}_{01}-(2\ell-1)\tilde{a}_{10}}
  {a_0(2\ell+1)}.
\end{equation*}

Assume that there exists a near-identity transformation \eqref{nips} with $N=m$
changing system~\eqref{pcg} into
\begin{equation*}
  \left(\begin{array}{c}
      \dot{x} \\
      \dot{y} \\
    \end{array}
  \right)=
  \left(\begin{array}{c}
    a_0-\sum\limits_{k=1}^{m}\mu_{k+1} x^k \\
    -a_0x^{2\ell-1}+\sum\limits_{k=1}^{m}\mu_{k+1} x^{k+2\ell-1} \\
  \end{array}\right)+\left(
    \begin{array}{c}
      \sum\limits_{i=0}^{\lfloor\frac{m+1}{2\ell}\rfloor}\hat{a}_{m+1-2\ell i,i}x^{m+1-2\ell i}y^i\\
      \sum\limits_{i=0}^{1+\lfloor\frac{m}{2\ell}\rfloor}\hat{b}_{m+2\ell-2\ell i,i}x^{m+2\ell-2\ell i}y^i \\
    \end{array}
  \right)+h.o.t..
\end{equation*}
Note that $\tilde{\Delta}^{m+1}$ has a basic consisting of vectors
$(x^{m+2-2\ell i}y^i,0)^\top$ for $0\le i\le q$ and vectors $(0,x^{m+2\ell+1-2\ell i}y^i)$ for $1\le i \le 1+r$, where $r:=\lfloor(m+1)/(2\ell)\rfloor,q:=\lfloor(m+2)/(2\ell)\rfloor$.
It is not hard to obtain
\begin{equation*}
  ad_p^m
  \left(
    \begin{array}{c}
      x^{m+2-2\ell i}y^i \\
      0 \\
    \end{array}
  \right)=
  \left(
    \begin{array}{c}
      (m+2-2\ell i)a_0x^{m+1-2\ell i}y^i-ia_0x^{m+1+2\ell-2\ell i}y^{i-1}\\
      (2\ell-1)a_0x^{m+2\ell-2\ell i}y^i \\
    \end{array}
  \right)
\end{equation*}
and
\begin{equation*}
  ad_p^m
  \left(
    \begin{array}{c}
      0 \\
      x^{m+2\ell+1-2\ell i}y^i \\
    \end{array}
  \right)=
  \left(
    \begin{array}{c}
      0\\
      (m+1+2\ell-2\ell i)a_0x^{m+2\ell-2\ell i}y^i-ia_0x^{m+4\ell-2\ell i}y^{i-1} \\
    \end{array}
  \right).
\end{equation*}
Associate with linearity of $ad_p^m$, we get
\begin{equation*}
  ad_p^m{\boldsymbol g}^{m+1}(x,y)=\sum_{i=0}^{q}\alpha_i ad_p^m
  \left(
    \begin{array}{c}
      x^{m+2-2\ell i}y^i \\
      0 \\
    \end{array}
  \right)+
  \sum_{i=1}^{r+1}\beta_iad_p^m
   \left(
    \begin{array}{c}
      0 \\
      x^{m+2\ell+1-2\ell i}y^i \\
    \end{array}
  \right)=
  \left(
    \begin{array}{c}
      \Gamma_1 \\
      \Gamma_2 \\
    \end{array}
  \right)
\end{equation*}
where
\begin{align*}
  \Gamma_1:=&\sum_{i=0}^{q}{\rm sgn}(q\!-\!i)\!\left[(m\!+\!2\!-\!2\ell i)\alpha_i\!-\!(i\!+\!1)\alpha_{i+1}\right]a_0 x^{m+1-2\ell i}y^i\!+\!(m\!+\!2\!-\!2\ell q)\alpha_{q}a_0x^{m+1-2\ell q}y^q,\\
  \Gamma_2:=&\left[(2\ell\!-\!1)\alpha_0\!-\!\beta_1\right]a_0x^{m+2\ell}\!+\!
  \sum_{i=0}^{r}{\rm sgn}(i)\left[(m\!+\!1\!+\!2\ell\!-\!2\ell i)\beta_i\!-\!(i\!+\!1)\beta_{i+1}\right]a_0x^{m+2\ell-2\ell i}y^i\\
    &+\sum_{i=0}^{q}{\rm sgn}(i)(2\ell\!-\!1)\alpha_ia_0x^{m+2\ell-2\ell i}y^i\!+\!(m\!+\!1\!-\!2\ell r)\beta_{r+1}a_0x^{m-2\ell r}y^{r+1}
\end{align*}
and ${\rm sgn}(x)$
is sign function.
We consider following two cases.

Case (I): $r=0$.
By \eqref{ppe}, we get an algebraic system $\{{\rm sgn}(q)(2\ell-1)a_0\alpha_1+(m+1)a_0\beta_1=\hat{b}_{m,1},
  \hat{a}_{m+1,0}-(m+2)a_0\alpha_0+{\rm sgn}(q)a_0\alpha_1=-\mu_{m+2},
  \hat{b}_{m+2\ell,0}-(2\ell-1)a_0\alpha_0+a_0\beta_1
  =\mu_{m+2}\}$
in variables $\alpha_0, \alpha_1, \beta_1, \mu_{m+2}$, which has a solution
\begin{align*}
  \alpha_0&=\frac{(m+1)(\hat{a}_{m+1,0}+\hat{b}_{m+2\ell,0})+\hat{b}_{m,1}}
  {a_0(m+1)(m+1+2\ell)},~~~~\alpha_1=0,~~~~
  \beta_1=\frac{\hat{b}_{m,1}}{a_0(m+1)},\\
  \mu_{m+2}&=\frac{(1-2\ell)(m+1)\hat{a}_{m+1,0}+(m+1)(m+2)\hat{b}_{m+2\ell,0}+
  (m+2)\hat{b}_{m,1}}{(m+1)(m+1+2\ell)}.
\end{align*}

Case (II): $r\ge1$.
For the first subcase: $2\ell \mid (m+2)$.
Then we get $q=r+1$ and $r=\lfloor m/(2\ell)\rfloor$.
By \eqref{ppe}, we get an algebraic system $EQ_1$ consisting of $(2\ell-1)a_0\alpha_q+(m+1-2\ell r)a_0\beta_{r+1}=\hat{b}_{m-2\ell r,r+1}$ and
\begin{equation}\label{m1q}
\left\{\begin{aligned}
  &(m+2-2\ell i)a_0\alpha_i-{\rm sgn}(q-i)(i+1)a_0\alpha_{i+1}=\hat{a}_{m+1-2\ell i,i},~i=1,...,r,\\
  &(2\ell-1)a_0\alpha_i+(m+1+2\ell-2\ell i)a_0\beta_i-(i+1)a_0\beta_{i+1}=\hat{b}_{m+2\ell-2\ell i,i},~i=1,...,r,\\
  &\hat{a}_{m+1,0}-(m+2)a_0\alpha_0+a_0\alpha_1=-\mu_{m+2},\\
  &\hat{b}_{m+2\ell,0}-(2\ell-1)a_0\alpha_0+a_0\beta_1=\mu_{m+2},
\end{aligned}\right.
\end{equation}
in variables $\alpha_0, \alpha_1,..., \alpha_q, \beta_1,..., \beta_{r+1}, \mu_{m+2}$.
Clearly,
\begin{equation*}
  i\le q=\frac{m+2}{2\ell}=\biggl\lfloor\frac{m+1}{2\ell}\biggr\rfloor+1<\frac{m+1}{2\ell}+1
\end{equation*}
implying $m+2-2\ell i>0$ for $1\le i \le q-1$ and $m+1+2\ell -2\ell i>0$ for $1\le i \le r+1$.
By straight computations, we get a solution of $EQ_1$ given by
\begin{equation*}
  \alpha_q=0,~~~~~\beta_{r+1}=\frac{\hat{b}_{m-2\ell r,r+1}}{a_0(m+1-2\ell r)}
\end{equation*}
and
\begin{equation}\label{ppsl}
\begin{aligned}
 \alpha_i&=\frac{\hat{a}_{m+1-2\ell i,i}+{\rm sgn}(q-i)(i+1)\alpha_{i+1}}{a_0(m+2-2\ell i)},~~~i=1,...,r,\\
 \beta_i&=\frac{\hat{b}_{m+2\ell-2\ell i,i}+(i+1)\beta_{i+1}-(2\ell-1)\alpha_i}{a_0(m+1+2\ell-2\ell i)},~~~i=1,...,r,\\
 \alpha_0&=\frac{\hat{a}_{m+1,0}+\hat{b}_{m+2\ell,0}+a_0(\alpha_1+\beta_1)}
 {a_0(m+1+2\ell)},\\
\mu_{m+2}&=\hat{b}_{m+2\ell,0}-(2\ell-1)a_0\alpha_0+a_0\beta_1.
\end{aligned}
\end{equation}

For the second subcase: $2\ell \nmid (m+2)$ and $2\ell \mid (m+1)$. Then we get $q=r$ and $r=\lfloor m/(2\ell)\rfloor+1$.
By \eqref{ppe}, we get an algebraic system $EQ_2$ as \eqref{m1q} in variables $\alpha_0, \alpha_1,..., \alpha_q, \beta_1,..., \beta_{r+1}, \mu_{m+2}$.
Since
\begin{equation}\label{seq2}
  i\le q=\biggl\lfloor\frac{m+2}{2\ell}\biggr\rfloor<\frac{m+2}{2\ell}
\end{equation}
implying $m+2-2\ell i>0$ for $1\le i \le q$ and
\begin{equation*}
  i\le r=\biggl\lfloor\frac{m}{2\ell}\biggr\rfloor+1<\frac{m+1}{2\ell}+1
\end{equation*}
implying $m+1+2\ell -2\ell i>0$ for $1\le i \le r$,
it is not hard to obtain a solution of $EQ_2$ given by $\beta_{r+1}=0$ and \eqref{ppsl}.

For the third subcase: $2\ell \nmid (m+2)$ and $2\ell \nmid (m+1)$. Then we get $q=r=\lfloor m/(2\ell)\rfloor$.
By \eqref{ppe}, we get an algebraic system $EQ_3$ consisting of $(m+1-2\ell r)a_0\beta_{r+1}=\hat{b}_{m-2\ell r,r+1}$ and \eqref{m1q} in variables $\alpha_0, \alpha_1,..., \alpha_q, \beta_1,..., \beta_{r+1}, \mu_{m+2}$.
From \eqref{seq2} and
\begin{equation*}
  i\le r+1=\biggl\lfloor\frac{m+1}{2\ell}\biggr\rfloor+1<\frac{m+1}{2\ell}+1
\end{equation*}
implying $m+1+2\ell -2\ell i>0$ for $1\le i\le r+1$,
we uniquely obtain
\begin{equation*}
  \beta_{r+1}=\frac{\hat{b}_{m-2\ell r,r+1}}{a_0(m+1-2\ell r)}
\end{equation*}
and \eqref{ppsl}.
Letting $p_{m+2-2\ell i,i}:=\alpha_i,i=0,...,q$ and $q_{m+1+2\ell-2\ell i,i}:=\beta_i,i=1,...,r+1$, we find a near-identity
\begin{equation}\label{nilp}
  \left(\begin{array}{c}
      x \\
      y \\
    \end{array}
  \right)\to\left(
   \begin{array}{c}
      x \\
      y \\
    \end{array}
  \right)
  +\sum_{k=1}^{m+1}
  \left(
    \begin{array}{c}
     \sum\limits_{i=0}^{\lfloor\frac{k+1}{2\ell}\rfloor}p_{k+1-2\ell i,i}x^{k+1-2\ell i}y^i\\
      \sum\limits_{i=1}^{1+\lfloor\frac{k}{2\ell}\rfloor}q_{k+2\ell-2\ell i,i}x^{k+2\ell-2\ell i}y^i \\
    \end{array}
  \right)+{\boldsymbol R}_{m+2}(x,y)
\end{equation}
changing system~\eqref{pcg} into
\begin{equation}\label{nflp}
  \left(\begin{array}{c}
      \dot{x} \\
      \dot{y} \\
    \end{array}
  \right)=
  \left(\begin{array}{c}
    a_0-\sum\limits_{k=1}^{m+1}\mu_{k+1} x^k \\
    -a_0x^{2\ell-1}+\sum\limits_{k=1}^{m+1}\mu_{k+1} x^{k+2\ell-1} \\
  \end{array}\right)+{\boldsymbol R}_{m+1}(x,y)
\end{equation}
in a small neighbourhood of $O$.
Note that terms in ${\boldsymbol R}_{m+2}(x,y)$ in \eqref{nilp} do not
influence the terms of quasi-homogeneous degree no more than $m$ in \eqref{nflp}.
It means that one can get \eqref{npN} from \eqref{pcg} by the $(m+1)$-th order truncation of \eqref{nilp}.
Therefore, the claim below system~\eqref{pcg} is proved by the mathematical induction.

Combining linear transformation~\eqref{tl} with near-identity transformation~\eqref{nips},
we find a diffeomorphism~\eqref{dfps} with
$r_{k+1-2\ell i,i}=p_{k+1-2\ell i,i}$($i=0,...,q$) and $s_{k+1-2\ell i,i}\!\!=\!\!-b_{2\ell-1,0}q_{k+1-2\ell i,i}/a_0$ ($i=1,...,r+1$).
This diffeomorphism rewrites the upper subsystem of \eqref{ge} as system~\eqref{npN}.
This completes the proof in Step 1.

{\it Step 2: transform \eqref{npN} into \eqref{nfps}.}
By direct calculation, we obtain that the time scaling~\eqref{tsps} rewrites system~\eqref{npN} as
\begin{equation*}
  \left(\begin{array}{c}
      \dot{x} \\
      \dot{y} \\
    \end{array}
  \right)=
  \left(\begin{array}{c}
    -1-\sum\limits_{k=1}^{N}\tilde{\mu}_kx^k \\
    x^{2\ell-1}+\sum\limits_{k=1}^{N}\tilde{\mu}_kx^{k+2\ell-1}  \\
  \end{array}\right)+\tilde{\boldsymbol R}_{N}(x,y),
\end{equation*}
where $\tilde{\mu}_1:=-T_1-\mu_2/a_0$ and
\begin{equation*}
  \tilde{\mu}_k:=-T_k-\frac{1}{a_0}\left(\mu_{k+1}+\sum\limits_{i=1}^{k-1}T_{k-i}\mu_{i+1}\right),~~~~k=2,...,N.
\end{equation*}
Solving $\tilde{\mu}_1=\cdots=\tilde{\mu}_N=0$,
we obtain
$$T_1=-\frac{\mu_2}{a_0},~~~~~T_k=-\frac{1}{a_0}\left(\mu_{k+1}+\sum_{i=1}^{k-1}T_{k-i}\mu_{i+1}\right), ~~k=2,...,N.$$
Thus time scaling \eqref{tsps} transforms \eqref{npN} into \eqref{nfps}.
This completes the proof in Step 2 and hence this lemma is proved.
\end{proof}

Before computing
the half return map in the following two lemmas,
we introduce
the following useful {\it ordinary Bell polynomials} (\cite{COM})
\begin{equation*}
  \hat{B}_{k,i}\left(v_1,v_2,...,v_{k-i+1}\right)=
  \sum_{S_{k-i+1}}\frac{i!}{b_1!b_2!\cdots b_{k-i+1}!}\prod_{j=1}^{k-i+1}v_j^{b_j},
  \end{equation*}
where
$S_{k-i+1}$ is the set of all $(k-i+1)$-tuples of nonnegative integers $(b_1,...,b_{k-i+1})$ satisfying $\sum_{j=1}^{k-i+1}jb_j=k$ and $\sum_{j=1}^{k-i+1}b_j=i$.
Notice that
\begin{equation*}
  {\left(\sum_{j=1}^{\infty}v_jx^j\right)}^n=\sum_{i=n}^{\infty}
  \hat{B}_{i,n}(v_1,...,v_{i-n+1})x^i.
\end{equation*}

\begin{lm}\label{lcofs}
In a small neighbourhood of $O$,
the upper half return map $\Pi^+(x)$ of FF normal form~\eqref{NFF} can be expanded as
\begin{equation*}
  \Pi^+(x)=-\sum_{k=1}^{N+1}u^+_k x^k+o\left(x^{N+1}\right)
\end{equation*}
and the coefficients $u^+_k$ are recursively given by
\begin{equation}\label{u+k}
u^+_1:=e^{\gamma^+_1 \pi},~~~~
  u^+_k:=\gamma^+_{k}C_k^{FF}\!+\!
   \Phi(\pi;\gamma^+_1,...,\gamma^+_{k-1})
   ~{\rm for}~ k\in \{2,...,N+1\},
   \end{equation}
where $C_k^{FF}$ is defined in Theorem~\ref{Lya-FFPP},
\begin{equation*}
 \Phi(\theta;\gamma^+_1,...,\gamma^+_{k-1}):=\!\!
 \left\{
 \begin{aligned}
 &0 &&{\rm if}~k=2,\\
 &\sum_{i=2}^{k-1}\!\gamma^+_{i}\!\!\int_{0}^{\theta}\!\!\sin^{i-1}\tau
  e^{\gamma^+_1(\theta-\tau)} \hat{B}_{k,i}\left(r_1(\tau),r_2(\tau),...,r_{k-i+1}(\tau)\right)d\tau~ &&{\rm if}~k>2.
  \end{aligned}
  \right.
\end{equation*}
Here $\hat{B}_{k,i}$ is ordinary Bell polynomial, $r_1(\tau)=e^{\gamma^+_1 \tau}$ and
\begin{equation*}
  r_j(\tau)=\gamma^+_{j}e^{\gamma^+_1\tau}\int_{0}^{\tau}e^{(j-1)\gamma^+_1 s}\sin^{j-1}
  s ds+\Phi(\tau;\gamma^+_1,...,\gamma^+_{j-1}),~~~~
  j=2,...,k-1.
\end{equation*}
\end{lm}

\begin{proof}
By $x=r\cos\theta, y=r\sin\theta$, we write the upper subsystem of FF normal form~\eqref{NFF} as polar coordinates form and obtain equation
\begin{equation}\label{pcfs}
  \frac{dr}{d\theta}=\gamma^+_1 r + \sum_{k=2}^{N+1}\gamma^+_{k} \sin^{k-1}\theta r^{k}+o(r^{N+1}).
\end{equation}
From the analyticity of the right hand side function of \eqref{pcfs},
we get that the solution $r(\theta,\rho)$ with $r(0,\rho)=\rho$ of \eqref{pcfs} can be expanded as $r(\theta,\rho)=\sum_{k=1}^{+\infty}r_k(\theta)\rho^k$
in neighborhoods of $\rho=0$.
By $r(\theta)=R(\theta)e^{\gamma^+_1 \theta}$,
we eliminate the linear term in \eqref{pcfs} and further obtain
  \begin{equation}\label{ceoR}
    \frac{dR}{d\theta}=\sum_{k=2}^{N+1}\gamma^+_{k}\sin^{k-1}\theta
    e^{(k-1)\gamma^+_1  \theta}R^k+o(R^{N+1}).
  \end{equation}
It is not hard to obtain that equation~\eqref{ceoR} has solution
\begin{equation}\label{slR}
  R(\theta,\rho)=\sum_{k=1}^{+\infty}R_k(\theta)\rho^k,
\end{equation}
where $ R(0,\rho)=\rho, R_k(\theta)=e^{-\gamma^+_1 \theta}r_k(\theta)$.

Substituting \eqref{slR} into two sides of equation~\eqref{ceoR}, we get
\begin{align*}
 \sum_{k=1}^{+\infty}R'_k(\theta)\rho^k=&\sum_{k=2}^{N+1}
 \gamma^+_{k}\sin^{k-1}\theta
    e^{(k-1)\gamma_1^+\theta}\left(\sum_{n=1}^{+\infty}R_n(\theta)
    \rho^n\right)^k+o(\rho^{N+1})\\
    =&\sum_{k=2}^{N+1}\gamma^+_{k}\sin^{k-1}\theta e^{(k-1)\gamma_1^+ \theta}\sum_{n=k}^{+\infty}\hat{B}_{n,k}
 \left(R_1(\theta),R_2(\theta),...,R_{k-i+1}(\theta)\right)\rho^n
 +o(\rho^{N+1})\\
 =&\sum_{k=2}^{N+1}\left(\sum_{i=2}^{k}\gamma^+_{i}\sin^{i-1}\theta e^{(i-1)\gamma_1^+\theta}\hat{B}_{k,i}
 \left(R_1(\theta),R_2(\theta),...,R_{k-i+1}(\theta)\right)\right)
 \rho^k+o(\rho^{N+1}).
\end{align*}
 Comparing the coefficients of $\rho^k$, we obtain $R'_1(\theta)=0$ and
 \begin{equation*}
   R'_k(\theta)=\sum_{i=2}^{k}\gamma^+_{i}\sin^{i-1}\theta e^{(i-1)\gamma_1^+ \theta}\hat{B}_{k,i}
 \left(R_1(\theta),R_2(\theta),...,R_{k-i+1}(\theta)\right),~~~k=2,...,N+1.
 \end{equation*}
By the initial condition $r(0,\rho)=\rho$ implying $R_1(0)=r_1(0)=1, R_k(0)=r_k(0)=0,k\ge2$ and above equations,
 we get that $r_1(\theta)=e^{\gamma_1^+ \theta}$ and
  for $2\le k\le N+1$,
 \begin{align*}
   r_k(\theta)\!\!=&e^{\gamma_1^+ \theta}\sum_{i=2}^{k}\int_{0}^{\theta}\gamma^+_{i}\sin^{i-\!1}\tau e^{(i-\!1)\gamma_1^+ \tau }\hat{B}_{k,i}
 \left(R_1(\tau ),R_2(\tau ),...,R_{k-i+1}(\tau )\right)d\tau\\
 =&e^{\gamma_1^+ \theta}\!\left(\!\int_{0}^{\theta}\gamma^+_{k}\sin^{k\!-\!1}\tau e^{(k-\!1)\gamma_1^+ \tau }d\tau\!+\!
 \sum_{i=2}^{k\!-\!1}\int_{0}^{\theta}\!\gamma^+_{i}\sin^{i\!-\!1}\tau e^{(i\!-\!1)\gamma_1^+ \tau }\hat{B}_{k,i}
 \left(R_1(\tau ),...,R_{k-i+1}(\tau )\right)d\tau\!\right)\\
 =&\gamma^+_{k}e^{\gamma_1^+ \theta}\int_{0}^{\theta}\sin^{k\!-\!1}\tau e^{(k\!-\!1)\gamma_1^+ \tau }d\tau\!+\!
 \sum_{i=2}^{k-1}\gamma^+_{i}\int_{0}^{\theta}\sin^{i\!-\!1}\tau e^{\gamma_1^+ (\theta-\tau) }\hat{B}_{k,i}
 \left(r_1(\tau ),...,r_{k-i+1}(\tau )\right)d\tau,
 \end{align*}
where $\hat{B}_{k,i}
 (e^{-\gamma_1^+\tau}r_1(\tau ),...,e^{-\gamma_1^+\tau}r_{k-i+1}(\tau ))=e^{-i\gamma_1^+\tau}\hat{B}_{k,i}
 (r_1(\tau ),...,r_{k-i+1}(\tau ))$ is used.
Associating  with
$$
  \int_{0}^{\pi}\sin^{k-1}\tau e^{(k-1)\gamma_1^+ \tau }d\tau=e^{-\gamma_1^+\pi} C_k^{FF},~~k\ge2
$$
and $\Pi^+(x)=-r(\pi,x)=-\sum_{k=1}^{+\infty}r_k(\pi)x^k$, this lemma is proved.
 \end{proof}

\begin{lm}\label{lcops}
In a small neighbourhood of $O$,
the upper half return map $\Pi^+(x)$ of PP normal form~\eqref{NPP} can be expanded as
\begin{equation}
  \Pi^+(x)=-\sum_{k=1}^{N+1}v^+_k x^k+o(x^{N+1}),
\label{rtmap}
\end{equation}
and the coefficients $v^+_k$ are recursively given by $v^+_1:=1$ and
\begin{equation*}
v^+_2:=\frac{2\sigma_2^+}{1+2\ell^+},~~  v^+_k:=\left(1+(-1)^{k}\right) \frac{\sigma^+_{k}}{k+2\ell^+-1}+
 \Psi(\sigma^+_2,...,\sigma^+_{k-1})~{\rm for}~3\le k\le N+1,
\end{equation*}
where
$$
  \Psi(\sigma^+_2,...,\sigma^+_{k-1})\!:=\!
  -\hat{B}_{k+2\ell^+\!-\!1,2\ell^+}(\!v^+_1\!,...,v^+_{k-1}\!,0)\!-\!\sum_{i=2\ell^++1}
  ^{k+2\ell^+\!-\!2}\!\!(\!-1\!)^i\frac{\sigma^+_{i-2\ell^+\!+\!1}}{i}\hat{B}_{k+2\ell^+\!-\!1,i}
  (\!v^+_1\!,...,v^+_{k+2\ell^+\!-i}).
$$
\end{lm}

\begin{proof}
  Since the PP normal form can be obtained for any positive integer $N$,
  for a given $N$ there exists sufficiently large $M>N$ such that the upper subsystem of PP normal form~\eqref{NPP} becomes
  \begin{equation}\label{lypp}
    \left(\begin{array}{c}
      \dot{x} \\
      \dot{y} \\
    \end{array}
  \right)=
  \left(\begin{array}{c}
    -1 \\
    x^{2\ell^+-1}+\sum\limits_{k=1}^{M}\sigma^+_{k+1} x^{k+2\ell^+-1}
  \end{array}\right)+{\boldsymbol R}^+_{M}(x,y),
  \end{equation}
 where ${\boldsymbol R}^+_{M}(x,y):=(R^+_{M+1}(x,y), R^+_{M+2\ell^+}(x,y))^\top$ and
 $R^+_{M+1}(x,y)$ (resp. $R^+_{M+2\ell^+}(x,y)$) is the sum of quasi-homogeneous polynomials of type $(1,2\ell^+)$ and degree no less than $M+1$ (resp. $M+2\ell^+$).
By time scaling $dt\to dt/(1+R^+_{M+1}(x,y))$,
system~\eqref{lypp} is transformed into
\begin{equation}
    \left(\begin{array}{c}
      \dot{x} \\
      \dot{y} \\
    \end{array}
  \right)=
  \left(\begin{array}{c}
    -1 \\
    x^{2\ell^+-1}+\sum\limits_{k=1}^{M}\sigma^+_{k+1} x^{k+2\ell^+-1}+R^+_{M+2\ell^+}(x,y)
  \end{array}\right).
  \label{NSsysform}
  \end{equation}

From analyticity of $R^+_{M+2\ell^+}(x,y)$, we write
$
  R^+_{M+2\ell^+}(x,y)=h(x,0)+yg(x,y),
$
where analytic functions $h(x,0)$ and $g(x,y)$ are the sums of quasi-homogeneous polynomial with
degree no less than $M+2\ell^+$ and $M$ respectively.
By \cite[Theorem C]{DD}, the upper half return map $\Pi^+(x)$ of system~\eqref{NSsysform}
is of form \eqref{rtmap} and $v_1^+=1$, $v^+_2=2\sigma_2^+/(1+2\ell^+)$ and for $3\le k\le N+1$,
$$
  v_k^+=\mu^+_{k+2\ell^+-1}-\mu^+_{2\ell^+}\hat{B}_{k+2\ell^+-1,2\ell^+}
  (v^+_1,...,v^+_{k-1},0)-
  \sum\limits_{i=2\ell^++1}^{k+2\ell^+-1}
 (-1)^{i}\mu^+_{i}\hat{B}_{k+2\ell^+-1,i}(v^+_1,...,v^+_{k+2\ell^+-i}),
 $$
 where $\mu^+_{2\ell^+}=1/(2\ell^+)$ and for $i=2\ell^++1,...,2\ell^++N$
\begin{equation*}
  \mu^+_i=\frac{1}{i!}\sum_{j=1}^{i}(-1)^j
  \left(
  \begin{array}{c}
  i \\
  j \\
   \end{array}
   \right)
  \frac{d^{i-1}x^{2\ell^+}f(x)}{dx^{i-1}}\bigg|_{x=0}=
  \frac{\sigma^+_{i-2\ell^++1}}{i}
\end{equation*}
by  \cite[Formula (10)]{DD}. Here
\begin{equation*}
  f(x):=\sum_{k=1}^{M}\sigma^+_{k+1} x^{k-1}+x^{-2\ell^+}h(x,0).
\end{equation*}
Substituting $\mu^+_i$ into the formula of $v_k^+$, we further get
 \begin{align*}
v_k^+ =&\left(1-(-1)^{k-1}\hat{B}_{k+2\ell^+-1,k+2\ell^+-1}(v^+_1)\right)
 \frac{\sigma^+_{k}}{k+2\ell^+-1}+
 \Psi(\sigma^+_2,...,\sigma^+_{k-1})\\
 =&\left(1+(-1)^k\right) \frac{\sigma^+_{k}}{k+2\ell^+-1}+
 \Psi(\sigma_2^+,...,\sigma^+_{k-1}),
\end{align*}
where $\Psi(\sigma^+_2,...,\sigma^+_{k-1})$ is given in the statement of this lemma.
The proof is finished.
\end{proof}

%%%%%%%%%%%%%%%%%%%%%%%%%%%%%%%%%%%%%%%%%%%%%%%%%%%%%%%%%%%%%%%%%

%%%%%%%%%%%%%%%%%%%%%%%%%%%%%%%%%%%%%%%%%%%%%%%%%%%%%%%%%%%%%%%%%

\section{Proofs of main results}
\setcounter{equation}{0}
\setcounter{lm}{0}
\setcounter{thm}{0}
\setcounter{rmk}{0}
\setcounter{df}{0}
\setcounter{cor}{0}

By the preliminary lemmas given in section 2, we can prove our main results Theorems~\ref{NFFFPP}-\ref{OFF} as follows.

\begin{proof}[Proof of Theorem~\ref{NFFFPP}]
Since three normal forms are given for three cases respectively, this proof is split into three steps.

{\it Step 1: obtain FF normal form~\eqref{NFF}.}
Suppose that system~\eqref{ge} has a monodromic singular point $O$ of FF type.
By Lemma~\ref{niofs}, there exists a diffeomorphism
\begin{equation}\label{unifs}
  \left(\begin{array}{c}
      x \\
      y \\
    \end{array}
  \right)\to
   \left(
     \begin{array}{cc}
       1 & q^+_2 \\
       0 & q^+_1 \\
     \end{array}
   \right)
  \left(
   \begin{array}{c}
      x \\
      y \\
    \end{array}
  \right)+\sum_{k=2}^{N+1}
  \left(\begin{array}{c}
      \sum\limits_{i=0}^{k}r^+_{k-i,i}x^{k-i} y^{i} \\
      \sum\limits_{i=1}^{k}s^+_{k-i,i}x^{k-i} y^{i} \\
    \end{array}\right)
\end{equation}
and a time scaling $dt\to \left(1-\sum_{k=1}^{N}T^+_ky^k\right)dt/\beta^+ $
rewriting the upper subsystem of \eqref{ge} as the upper one of \eqref{NFF}.
We denote transformation $(x,y)^\top\to(x,-y)^\top$ by $T$ and
\begin{equation*}
  \left(\begin{array}{c}
      x \\
      y \\
    \end{array}
  \right)\to
   \left(
     \begin{array}{cc}
       1 & q^-_2 \\
       0 & q^-_1 \\
     \end{array}
   \right)
  \left(
   \begin{array}{c}
      x \\
      y \\
    \end{array}
  \right)+\sum_{k=2}^{N+1}
  \left(\begin{array}{c}
      \sum\limits_{i=0}^{k}r^-_{k-i,i}x^{k-i} y^{i} \\
      \sum\limits_{i=1}^{k}s^-_{k-i,i}x^{k-i} y^{i} \\
    \end{array}\right)
\end{equation*}
by $T_1$.
It is not hard to obtain there exists a diffeomorphism $T\circ T_1 \circ T$ and a time scaling $dt\to \left(1-\sum_{k=1}^{N}T^-_ky^k\right)dt/\beta^- $
rewriting the lower subsystem of \eqref{ge} as the lower one of \eqref{NFF} by Lemma~\ref{niofs}.
From the proof of Lemma~\ref{niofs} we get that $r_{k,0}$ ($k=2,...,N+1$) in \eqref{dhfs} can be arbitrarily chosen.
Thus, let $r^{\pm}_{k,0}=0$ for $k=2,...,N+1$ and we find a transformation
\begin{equation}\label{whff}
  \left(\begin{array}{c}
      x \\
      y \\
    \end{array}
  \right)\to\left\{\begin{aligned}
   &\left(
     \begin{array}{cc}
       1 & q^+_2 \\
       0 & q^+_1 \\
     \end{array}
   \right)
  \left(
   \begin{array}{c}
      x \\
      y \\
    \end{array}
  \right)+\sum_{k=2}^{N+1}
  \left(\begin{array}{c}
      \sum\limits_{i=1}^{k}r^+_{k-i,i}x^{k-i} y^{i} \\
      \sum\limits_{i=1}^{k}s^+_{k-i,i}x^{k-i} y^{i} \\
    \end{array}\right)~~&&{\rm for~}y\ge0,\\
    &\left(
     \begin{array}{cc}
       1 & q^-_2 \\
       0 & q^-_1 \\
     \end{array}
   \right)
  \left(
   \begin{array}{c}
      x \\
      y \\
    \end{array}
  \right)+\sum_{k=2}^{N+1}
  \left(\begin{array}{c}
      \sum\limits_{i=1}^{k}(-1)^ir^-_{k-i,i}x^{k-i} y^{i} \\
      \sum\limits_{i=1}^{k}(-1)^{i+1}s^-_{k-i,i}x^{k-i} y^{i} \\
    \end{array}\right)&&{\rm for~}y\le0\\
    \end{aligned}\right.
\end{equation}
and a time scaling
\begin{equation*}
  dt\to\left\{\begin{aligned}
  &\frac{1}{\beta^+}\left(1-\sum_{k=1}^{N}T^+_ky^k\right)dt~~&&{\rm for~}y\ge0,\\
  &\frac{1}{\beta^-}\left(1-\sum_{k=1}^{N}T^-_ky^k\right)dt &&{\rm for~}y\le0
  \end{aligned}\right.
\end{equation*}
changing system~\eqref{ge} into \eqref{NFF} in a small neighbourhood of $O$.
Since transformation~\eqref{whff} is a homeomorphism by Lemma~\ref{niih} and $\beta^+\beta^->0$,
we get that system~\eqref{ge} and \eqref{NFF} are topologically equivalent to each other near $O$.
This completes the proof in Step 1.

{\it Step 2: obtain FP normal form~\eqref{NFP}.} Suppose that system~\eqref{ge} has a monodromic singular point $O$ of FP type.
The upper subsystem of \eqref{NFP} can be similarly obtained as shown in proof of FF normal form in step 1.
We denote transformation
\begin{equation*}
  \left(\begin{array}{c}
      x \\
      y \\
    \end{array}
  \right)\to
  \left(
    \begin{array}{cc}
      1 & 0 \\
      0 & q^-_1 \\
    \end{array}
  \right)
  \left(
   \begin{array}{c}
      x \\
      y \\
    \end{array}
  \right)
  +\sum_{k=2}^{N+1}
  \left(
    \begin{array}{c}
     \sum\limits_{i=0}^{\lfloor\frac{k}{2\ell}\rfloor}r^-_{k-2\ell i,i}x^{k-2\ell i}y^i\\
      \sum\limits_{i=1}^{1+\lfloor\frac{k-1}{2\ell}\rfloor}s^-_{k+2\ell-1-2\ell i,i}x^{k+2\ell-1-2\ell i}y^i \\
    \end{array}
  \right)
\end{equation*}
by $T_2$.
According to Lemma~\ref{nfops}, there exists a diffeomorphism $T\circ T_2\circ T$ and a time scaling $dt\to \left(1-\sum_{k=1}^{N}T_kx^k\right)dt/a^-_0$ rewriting the lower subsystem of \eqref{ge} as the lower one of \eqref{NFP},
where $a^-_0=X^-(0,0)$.
Note that $r^+_{k,0}$ in \eqref{unifs} can be arbitrarily chosen, we let $r^+_{k,0}=r^-_{k,0}$ for $k=2,...,N+1$.
Therefore, transformation
\begin{equation}\label{whpf}
  \left(\begin{array}{c}
      x \\
      y \\
    \end{array}
  \right)\!\to\!\left\{\!\!\!\begin{aligned}
   \!&\left(\!\!
     \begin{array}{cc}
       1 & q^+_2 \\
       0 & q^+_1 \\
     \end{array}
  \! \!\right)\!\!
  \left(\!\!
   \begin{array}{c}
      x \\
      y \\
    \end{array}
 \! \!\right)\!+\!\sum_{k=2}^{N+1}
  \left(\!\!
  \begin{array}{c}
      r^-_{k,0}x^k+\sum\limits_{i=1}^{k}r^+_{k-i,i}x^{k-i} y^{i} \\
      \sum\limits_{i=1}^{k}s^+_{k-i,i}x^{k-i} y^{i} \\
    \end{array}
    \!\!\right)\!&&{\rm for~}y\ge0,\\
    &\left(\!\!
    \begin{array}{cc}
      1 & 0\\
      0 & q^-_1 \\
    \end{array}
 \! \!\right)\!\!
  \left(\!\!
   \begin{array}{c}
      x \\
      y \\
    \end{array}
 \!\! \right)
  \!+\!\sum_{k=2}^{N+1}
  \left(\!\!\!\!
    \begin{array}{c}
      r^-_{k,0}x^k+\sum\limits_{i=1}^{\lfloor\frac{k}{2\ell^-}\rfloor}
      \tilde{r}^-_{k-2\ell^- i,i}x^{k-2\ell^- i}y^i\\
      \sum\limits_{i=1}^{1+\lfloor\frac{k-1}{2\ell^-}\rfloor}\tilde{s}^-_{k-1
      -2\ell^- (i-1),i}x^{k-1
      -2\ell^- (i-1)}y^i
      \end{array}
     \!\! \!\!\right)&&{\rm for~}y\le0\\
    \end{aligned}\right.
\end{equation}
with $\tilde{r}^-_{k-2\ell^- i,i}:=(-1)^i{r}^-_{k-2\ell^- i,i}$ and $\tilde{s}^-_{k-1-2\ell^-(i-1),i}:=(-1)^{i+1}{s}^-_{k-1-2\ell^-(i-1),i}$ for $i=1,...,k$
and time scaling
\begin{equation*}
  dt\to\left\{\begin{aligned}
  &\frac{1}{\beta^+}\left(1-\sum_{k=1}^{N}T^+_ky^k\right)dt~~&&{\rm for~}y\ge0,\\
  &\frac{1}{a_0^-}\left(1-\sum_{k=1}^{N}T^-_ky^k\right)dt&&{\rm for~}y\le0
  \end{aligned}\right.
\end{equation*}
change system~\eqref{ge} into \eqref{NFP} in a small neighbourhood of $O$.
Since transformation~\eqref{whpf} is a homeomorphism by Lemma~\ref{niih} and $\beta^+a_0^->0$,
we get that system~\eqref{ge} and \eqref{NFP} are topologically equivalent to each other near $O$.
This completes the proof in Step 2.

{\it Step 3: obtain PP normal form~\eqref{NPP}.} Suppose that system~\eqref{ge} has a monodromic singular point $O$ of PP type.
Clearly, the lower subsystem of ~\eqref{NPP} can be similarly obtained as in the proof of FP normal form in step 2.
According to the proof of Lemma~\ref{nfops}, there exists a linear transformation $(x,y)^\top \to (x, q^+_1 y)^\top$
rewriting the upper subsystem of \eqref{ge} as
\begin{equation}\label{up2}
  \left(
    \begin{array}{c}
      \dot{x} \\
      \dot{y} \\
    \end{array}
  \right)=
  \left(
    \begin{array}{c}
      a_0^+ \\
      -a_0^+x^{2\ell^+-1} \\
    \end{array}
  \right)+\sum_{k\ge0}
  \left(
    \begin{array}{c}
      \sum\limits_{i=0}^{\lfloor\frac{k+1}{2\ell^+}\rfloor}
      \tilde{a}_{k+1-2\ell^+ i,i}x^{k+1-2\ell^+ i}y^i\\
      \sum\limits_{i=0}^{1+\lfloor\frac{k}{2\ell^+}\rfloor}
      \tilde{b}_{k+2\ell^+-2\ell^+ i,i}x^{k+2\ell^+-2\ell^+ i}y^i \\
    \end{array}
  \right),
\end{equation}
where $a_0^+=X^+(0,0)$.
We claim that there exists a near-identity
\begin{equation*}
  \left(\begin{array}{c}
      x \\
      y \\
    \end{array}
  \right)\to\left(
   \begin{array}{c}
      x \\
      y \\
    \end{array}
  \right)
  +\sum_{k=2}^{N+1}
  \left(
    \begin{array}{c}
     r^-_{k,0}x^{k}+\sum\limits_{i=1}^{\lfloor\frac{k}{2\ell^+}\rfloor}
     p^+_{k-2\ell^+ i,i}x^{k-2\ell^+ i}y^i\\
      \sum\limits_{i=1}^{1+\lfloor\frac{k-1}{2\ell^+}\rfloor}q^+_{k-1
      -2\ell^+ (i-1),i}x^{k-1
      -2\ell^+ (i-1)}y^i \\
    \end{array}
  \right)
\end{equation*}
changing system~\eqref{up2} into
\begin{equation}\label{npNns}
  \left(\begin{array}{c}
      \dot{x} \\
      \dot{y} \\
    \end{array}
  \right)=
  \left(\begin{array}{c}
    a_0^+-\sum\limits_{k=1}^{N}\nu^+_{k+1} x^k \\
    -a_0^+x^{2\ell^+-1}+\sum\limits_{k=1}^{N}\eta^+_{k+1} x^{k+2\ell^+-1} \\
  \end{array}\right)+{\bf R}^+_N(x,y),
\end{equation}
where $r^-_{k,0}$ is exactly the $r^-_{k,0}$ determined in \eqref{whpf}.
In fact, this claim can be proved by finding ${\boldsymbol g}^{k+1}(x,y) \in \tilde{\Delta}^{k+1}$ and $\nu^+_{k+2},\eta^+_{k+2}$ to satisfy
\begin{equation}\label{ppe2}
  {\boldsymbol F}_k(x,y)-ad_p^k{\boldsymbol g}^{k+1}(x,y)=
  \left(
    \begin{array}{c}
      -\nu^+_{k+2} x^{k+1} \\
      \eta^+_{k+2} x^{k+2\ell^+} \\
    \end{array}
  \right)\in \Theta^k,~~k=0,...,N-1.
\end{equation}
For $N=1$, it is equivalent to solve algebraic system $\{\tilde{a}_{10}-2a_0^+r^-_{20}+\omega a^+_0p^+_{01}=-\nu^+_2,
\tilde{b}_{2\ell^+,0}- (2\ell^+-1)a_0^+r^-_{20}+a_0^+q^+_{11}=\eta^+_2,
\tilde{b}_{01}-a_0^+q^+_{11}-\omega a^+_0 p^+_{01}=0\}$
in variables $p^+_{01},q^+_{11},\nu^+_2,\eta^+_2$ by the proof of Lemma~\ref{nfops}, where $\omega:=\lfloor 1/\ell^+ \rfloor$.
Solving this system, we obtain $p^+_{01}=0, q^+_{11}=\tilde{b}_{01}/a_0^+$, $\nu^+_2=2a_0^+r^-_{20}-\tilde{a}_{10}$, $\eta^+_2=\tilde{b}_{2\ell^+,0}- (2\ell^+-1)a_0^+r^-_{20}+\tilde{b}_{01}$. Thus \eqref{ppe2} holds for $N=1$.
Assume that \eqref{ppe2} is true for $N=m$.
Let $p^+_{m+2-2\ell^+ i,i}:=\alpha_i$ for $i=1,...,q$ and $q^+_{m+1+2\ell^+-2\ell^+ i,i}:=\beta_i$ for $i=1,...,r+1$, where $r:=\lfloor(m+1)/(2\ell^+)\rfloor,q:=\lfloor(m+2)/(2\ell^+)\rfloor$.
We consider following two cases to prove our claim for $N=m+1$.
Case (I): $r=0$.
According to \eqref{ppe2},
we obtain an algebraic system
\begin{equation*}
  \left\{\begin{aligned}
  &{\rm sgn}(q)(2\ell^+-1)a^+_0\alpha_1+(m+1)a^+_0\beta_1=\hat{b}_{m,1},\\
  &\hat{a}_{m+1,0}-(m+2)a^+_0r^-_{m+2,0}=-\nu^+_{m+2},\\
  &\hat{b}_{m+2\ell^+,0}-(2\ell^+-1)a^+_0r^-_{m+2,0}+a^+_0\beta_1
  =\eta^+_{m+2}
\end{aligned}\right.
\end{equation*}
in variables $\alpha_1,\beta_1,\nu^+_{m+2},\mu^+_{m+2}$.
It is not hard to check that above system has a solution $\alpha_1=0$, $\beta_1=\hat{b}_{m,1}/((m+1)a^+_0)$ and
\begin{equation*}
  \nu^+_{m+2}=-\hat{a}_{m+1,0}+(m+2)a^+_0r^-_{m+2,0},~~~~
  \eta^+_{m+2}=\hat{b}_{m+2\ell^+,0}-(2\ell^+-1)a^+_0
  r^-_{m+2,0}+\frac{\hat{b}_{m,1}}{m+1}.
\end{equation*}
Case (II): $r\ge 0$.
As shown in the proof of Lemma~\ref{nfops},
we only need to solve algebraic systems $EQ'_j$ ($j=1,2,3$) in variables $\alpha_1,...,\alpha_q$, $\beta_1,...,\beta_{r+1}$,
$\nu^+_{m+2}, \eta^+_{m+2}$,
 where  $EQ'_j$ is $EQ_j$ with $\ell=\ell^+$ and the last two equations of \eqref{m1q} in $EQ_j$ are replaced by
\begin{equation*}
\hat{a}_{m+1,0}-(m+2)a_0^+r^-_{m+2,0}+a_0^+\alpha_1=-\nu^+_{m+2},~~~~
\hat{b}_{m+2\ell^+,0}-(2\ell^+-1)a_0^+r^-_{m+2,0}+a_0^+\beta_1=\eta^+_{m+2}.
\end{equation*}
It is not hard to obtain from $EQ'_j$ ($j=1,2,3$) that $\alpha_1,...,\alpha_q$, $\beta_1,...,\beta_{r+1}$ are exactly values given
in the proof of Lemma~\ref{nfops} in the solution of $EQ_j$ ($j=1,2,3$)
and
\begin{equation*}
  \nu^+_{m+2}=-\hat{a}_{m+1,0}+(m+2)a_0^+r^-_{m+2,0}-a_0^+\alpha_1,~~~~
  \eta^+_{m+2}=\hat{b}_{m+2\ell^+,0}-(2\ell^+-1)a_0^+r^-_{m+2,0}
  +a_0^+\beta_1.
\end{equation*}
Thus, \eqref{ppe2} is true for $N=m+1$. Therefore, our claim is true, i.e., system~\eqref{up2}
is equivalently changed into \eqref{npNns}.

Clearly, time scaling
  $dt\to dt/(-a_0^++\sum_{k=1}^{N}\nu^+_{k+1} x^k)$
changes the first component of the right-hand side of system~\eqref{npNns} into $-1+R^+_{N+1}(x,y)$.
And since
\begin{align*}
  \left(1-\sum\limits_{k=1}^{N}\frac{\nu^+_{k+1}}{a_0^+} x^k\right)^{-1}
  =&1+\sum_{l=1}^{+\infty}
  \left(\sum_{k=1}^{N}\frac{\nu^+_{k+1}}{a_0^+}x^k\right)^l\\
  =&1+\sum_{l=1}^{N}\sum_{k=l}^{N}\frac{\hat{B}_{k,l}(\nu^+_{2},...,
  \nu^+_{k-l+2})}{(a^+_0)^l}
  x^k+R^+_{N+1}(x,y)\\
  =&1+\sum_{k=1}^{N}\sum_{i=1}^{k}\frac{\hat{B}_{k,i}(\nu^+_{2},...,
  \nu^+_{k-i+2})}{(a^+_0)^i}
  x^k+\tilde{R}^+_{N+1}(x,y),
\end{align*}
we get that the second component becomes
\begin{align*}
 \frac{-a_0^+x^{2\ell^+-1}+\sum\limits_{k=1}^{N}\eta^+_{k+1} x^{k+2\ell^+-1}+R^+_{N+2\ell^+}(x,y)}
 {-a_0^++\sum\limits_{k=1}^{N}\nu^+_{k+1} x^k}
  =x^{2\ell^+-1}+\sum_{k=1}^{N}
  \sigma^+_{k+1} x^{k+2\ell^+-1}+\tilde{R}^+_{N+2\ell^+}(x,y),
\end{align*}
where $\eta^+_1:=-a_0^+$, $\sigma^+_2:=(\nu^+_2-\eta^+_2)/a_0^+$ and $$\sigma^+_{k+1}=-\frac{\eta^+_{k+1}}{a_0^+}-\sum_{j=1}^{k}
\sum_{i=1}^{j}\frac{\hat{B}_{j,i}
  (\nu^+_2,...,\nu^+_{j-i+2})\eta^+_{k-j+1}}{(a_0^+)^{i+1}},~~k=2,...,N.$$

From above analysis, we find a transformation
\begin{equation}\label{whpp}
  \left(\!\begin{array}{c}
      x \\
      y \\
    \end{array}
  \!\right)\!\!\to\!\!\left\{\!\!
 \!\! \begin{aligned}
  &\left(\!\!\!\begin{array}{cc}
      1 & 0 \\
      0 & q^+_1 \\
    \end{array}
  \!\!\!\right)\!\!\left(\!\!\!
   \begin{array}{c}
      x \\
      y \\
    \end{array}
 \!\!\! \right)
  \!+\!\sum_{k=2}^{N+1}
  \left(\!\!
    \begin{array}{c}
     r^-_{k,0}x^{k}+\sum\limits_{i=1}^{\lfloor\frac{k}{2\ell^+}\rfloor}
     r^+_{k-2\ell^+ i,i}x^{k-2\ell^+ i}y^i\\
      \sum\limits_{i=1}^{1+\lfloor\frac{k-1}{2\ell^+}\rfloor}s^+_{k-1
      -2\ell^+ (i-1),i}x^{k-1
      -2\ell^+ (i-1)}y^i \\
     \end{array}\!\!\right)&&{\rm for~}y\ge0,\\
     &\left(\!\!\!
    \begin{array}{cc}
      1 &  0 \\
      0 & q^-_1 \\
    \end{array}
  \!\!\right)\!\!
  \left(\!\!
   \begin{array}{c}
      x \\
      y \\
    \end{array}
  \!\!\!\right)
  \!+\!\sum_{k=2}^{N+1}
  \left(\!\!
    \begin{array}{c}
      r^-_{k,0}x^k+\sum\limits_{i=1}^{\lfloor\frac{k}{2\ell^-}\rfloor}
      \tilde{r}^-_{k-2\ell^- i,i}x^{k-2\ell^- i}y^i\\
      \sum\limits_{i=1}^{1+\lfloor\frac{k-1}{2\ell^-}\rfloor}
      \tilde{s}^-_{k-1
      -2\ell^- (i-1),i}x^{k-1
      -2\ell^- (i-1)}y^i \end{array}\!\!\right)&&{\rm for~}y\le0\\
    \end{aligned}\right.
\end{equation}
with ${r}^+_{k-2\ell^+ i,i}:=p^+_{k-2\ell^+ i,i}$ and ${s}^+_{k-1-2\ell^+ (i-1),i}:=q^+_1 q^+_{k-1-2\ell^+ (i-1),i}$ for $i=1,...,k$
and a time scaling
\begin{equation}
  dt\to\left\{\begin{aligned}
  &\frac{dt}{-a_0^++\sum\limits_{k=1}^{N}\nu^+_{k+1} x^k}~~&&{\rm for~}y\ge0,\\
  &\frac{1}{a_0^-}\left(1-\sum_{k=1}^{N}T^-_ky^k\right)dt&&{\rm for~}y\le0
  \end{aligned}\right.
  \label{tsfpp}
\end{equation}
changing system~\eqref{ge} into \eqref{NPP} in a small neighbourhood of $O$.
Since transformation~\eqref{whpp} is a homeomorphism by Lemma~\ref{niih} and $-a_0^+a_0^->0$,
we get system~\eqref{ge} and \eqref{NPP} are topologically equivalent to each other near $O$.
This completes the proof in Step 3.
\end{proof}

\begin{proof}[Proof of Theorem~\ref{Lya-FFPP}]
{\it Step 1: compute Lyapunov constants $V_k$ for FF normal form~\eqref{NFF}.}
Note that the upper half return map $\Pi^+(x)$ is obtained by Lemma~\ref{lcofs}.
By transformation $(x,y,t)\to (x,-y,-t)$,
the lower subsystem of \eqref{NFF} is changed into
\begin{equation*}
  \left(\begin{array}{c}
      \dot{x} \\
      \dot{y} \\
    \end{array}\right)
    =\left(\begin{array}{cc}
      -\gamma^-_1 & -1 \\
      1 &  -\gamma^-_1 \\
    \end{array}\right)
  \left(\begin{array}{c}
      x \\
      y \\
    \end{array}\right)
    +\sum_{i=1}^{N}(-1)^{i+1}{\gamma}^-_{i+1} y^i
    \left(\begin{array}{c}
      x \\
      y \\
    \end{array}\right)
    +G^-_{N+2}(x,y),
\end{equation*}
which is defined on the upper half plane.
Associating with Lemma~\ref{lcofs}, we get the lower half return map $(\Pi^-)^{-1}(x)=-\sum_{k=1}^{N+1}u_k^-x^k+o(x^{N+1})$,
where $u^-_1:=e^{-\gamma^-_1 \pi}$ and
\begin{equation}\label{u-k}
  u^-_k:=
 (-1)^{k}\gamma^-_{k}\tilde{C}_k^{FF}+
  \Phi(\pi; -\gamma^-_1,...,(-1)^{k-1}\gamma^-_{k-1})~~~{\rm for}~k=2,...,N+1.
\end{equation}
Here $\tilde{C}_k^{FF}$ is ${C}_k^{FF}$ with replacing $\gamma_1^+$ by $-\gamma_1^-$.
Then for displacement map $\Delta(x)$ in \eqref{DeltaF} we get $V_k=u^+_k-u^-_k$.
Clearly $V_1=e^{\gamma^+_1\pi}-e^{-\gamma^-_1\pi}$.
Moreover, by \eqref{u+k} and \eqref{u-k} we get that
$V_1=\cdots=V_{k-1}=0$ if and only if $\gamma^+_i+(-1)^{i-1}\gamma^-_i=0$ for all $i=1,...,k-1$,
which implies $\Phi(\pi;\gamma^+_1,...,\gamma^+_{k-1})=
\Phi(\pi;-\gamma^-_1,...,(-1)^{k-1}\gamma^-_{k-1})$.
Thus,
$V_k=u_k^+-u_k^-=(\gamma^+_k+(-1)^{k-1}\gamma^-_k)C_k^{FF}$ when $V_1=\cdots=V_{k-1}=0$.
This completes the proof in Step 1.

{\it Step 2: compute Lyapunov constants $V_k$ for FP normal form~\eqref{NFP}.}
For FP normal form~\eqref{NFP}, the upper half return map $\Pi^+(x)$ is given by Lemma~\ref{lcofs}.
By transformation $(x,y,t)\to(x,-y,-t)$, the lower subsystem of \eqref{NFP} is rewritten as
\begin{equation*}
  \left(\begin{array}{c}
      \dot{x} \\
      \dot{y} \\
    \end{array}
  \right)=
  \left(\begin{array}{c}
    -1 \\
    x^{2\ell^--1} \\
  \end{array}\right)+{\boldsymbol R}^-_{N}(x,y),
\end{equation*}
which is defined on the upper half plane.
Associating with $\hat{B}_{k+2\ell^--1,2\ell^-}(1,0,...,0)=0$ for $3\le k\le N+1$ and Lemma~\ref{lcops},
we obtain the lower half return map $(\Pi^-)^{-1}(x)=-x+o(x^{N+1})$.
Thus for displacement map $\Delta(x)$ in \eqref{DeltaF} we get $V_1=e^{\gamma^+_1\pi}-1$ and $V_k=u_k^+$ for $2\le k \le N+1$.
Clearly $V_1=0$ if and only if $\gamma^+_1=0$, thus we get $V_2=u_2^+=C_2^{FP}\gamma^+_2$ when $V_1=0$ by \eqref{u+k}.
Moreover, it is not hard to obtain that $V_1=\cdots =V_{k-1}=0$ if and only if $\gamma^+_i=0$, $i=1,...,k-1$.
Associated with $\Phi(\pi;0,...,0)=0$, we obtain $V_k=u_k^+=\gamma_k^+C_k^{FP}$ when $V_1=\cdots =V_{k-1}=0$.
This completes the proof in Step 2.

{\it Step 3: compute Lyapunov constants $V_k$ for PP normal form~\eqref{NPP}.}
For PP normal form~\eqref{NPP}, the upper half return map $\Pi^+(x)$ is given by Lemma~\ref{lcops}.
From last paragraph,
we get the lower half return map $(\Pi^-)^{-1}(x)=-x+o(x^{N+1})$.
Thus for displacement map $\Delta(x)$ in \eqref{DeltaF} we get $V_1=0$, $V_k=v_k^+$ for $2\le k \le N+1$,
where $v_k^+$ is given below \eqref{rtmap}.
Associated with $\hat{B}_{k+2\ell^+-1,i}(1,0,...,0)=0$ for $2\ell^+\le i\le k+2\ell-2$ and the expression of $v_k^+$,
we get $V_k=0$ for odd $k$ and $V_k=\sigma^+_k C_k^{PP}$ for even $k$ when
$V_1=\cdots=V_{k-1}=0$.
This completes the proof in Step 3.
This theorem is proved.
\end{proof}

\begin{proof}[Proof of Theorem~\ref{OFF}]
We extend the domain of the upper subsystem of \eqref{ge} with $O$ being a weak focus of order $\varsigma^+$
to whole plane. That is, in its FF normal form both the upper and the lower
ones have the form of the upper one of \eqref{NFF}.
By the expression of Lyapunov constant given in Theorem~\ref{Lya-FFPP}, it is not hard to obtain
that the first nonzero $\gamma^+_i$ is $\gamma^+_{2\varsigma^++1}$ when $0\le \varsigma^+<+\infty$,
and $\gamma^+_{2i+1}=0$ for any $i\in \mathbb{N}$ when $\varsigma^+=+\infty$.

{\it Step 1: consider the case that lower subsystem of \eqref{ge} has a weak focus of order $\varsigma^-\in \mathbb{N}\cup\{+\infty \}$.}
As the above paragraph, for its normal form we similarly get that
the first nonzero $\gamma^-_i$ is $\gamma^-_{2\varsigma^-+1}$ when $0\le \varsigma^-<+\infty$,
and $\gamma^-_{2i+1}=0$ for any $i\in \mathbb{N}$ when $\varsigma^-=+\infty$.
Thus, both $\gamma^+_1, \gamma^-_1$ are nonzero if $\varsigma^+=\varsigma^-=0$ and there is no any relation among
$\gamma_i^\pm$ ($i\ge 1$). By the expression of Lyapunov constant given in Theorem~\ref{Lya-FFPP},
the order $\varsigma$ of $O$ of system~\eqref{ge} can be any element in $\bigcup_{i=0}^{+\infty}\left\{i/2\right\}$.

If $\varsigma^+=0$ and $\varsigma^-\in \mathbb{Z}^+\cup\{+\infty\}$ then $\gamma^+_1\ne0$ and $\gamma^-_1=0$.
Associated with $\Delta(x)=\gamma_1^+x+o(x)$, we get $\varsigma=0$.
The result of the case $\varsigma^-=0$ and $\varsigma^+\in \mathbb{Z}^+\cup\{+\infty\}$ can be obtained similarly.
Conclusion {\bf (a)} is proved.

For the case of $\varsigma^{\pm}\in \mathbb{Z}^+\cup\{+\infty\}$ and $\varsigma^+\ne\varsigma^-$, letting $\varsigma^+<\varsigma^-$ without loss of generality.
Associated with Theorem~\ref{Lya-FFPP}, we get $\gamma^{\pm}_{2i+1}=0$ for $0\le i\le \varsigma^+-1$, $\gamma^+_{2\varsigma^++1}\ne0$ and $\gamma^-_{2\varsigma^++1}=0$,
and displacement function
\begin{equation}\label{df13b}
  \Delta(x)=\sum_{i=1}^{\varsigma^+}(\gamma_{2i}^+-\gamma_{2i}^-)x^{2i}+
  \gamma^+_{2\varsigma^++1}x^{2\varsigma^++1}+o(x^{2\varsigma^++1}).
\end{equation}
Thus, we obtain $\varsigma=(2i+1)/2$ for $0\le i\le \varsigma^+-1$ when
$\gamma^+_{2j}=\gamma^-_{2j}$ for all $j=1,...,i$ and $\gamma^+_{2i+2}\ne\gamma^-_{2i+2}$.
Additionally, $\varsigma=\varsigma^+$ when $\gamma^+_{2i}=\gamma^-_{2i}$ for all $1\le i\le \varsigma^+$ by \eqref{df13b}.
Conclusion {\bf (b)} is proved.

Clearly, $\gamma^{\pm}_{2i+1}=0$ ($i=0,...,\varsigma^+-1$) and $\gamma^+_{2\varsigma^++1}\gamma^-_{2\varsigma^++1}\ne0$ when $\varsigma^+=\varsigma^-\in \mathbb{Z}^+$.
By Theorem~\ref{Lya-FFPP}, we obtain
\begin{equation}\label{df13c}
  \Delta(x)=\sum_{i=1}^{\varsigma^+}(\gamma_{2i}^+-\gamma_{2i}^-)x^{2i}+
  \left(\gamma^+_{2\varsigma^++1}+\gamma^-_{2\varsigma^++1}\right)
  x^{2\varsigma^++1}+o(x^{2\varsigma^++1}).
\end{equation}
On the other hand,  $\gamma^{\pm}_{2i+1}=0$ for all $i\in \mathbb{N}$ when $\varsigma^+=\varsigma^-=+\infty$.
Then
 \begin{equation}\label{df13d}
  \Delta(x)=\sum_{i=1}^{+\infty}(\gamma_{2i}^+-\gamma_{2i}^-)x^{2i}
\end{equation}
by Theorem~\ref{Lya-FFPP}.
Same as proofs of conclusions {\bf (a)} and {\bf (b)}, conclusion {\bf (c)} can be proved associated with \eqref{df13c} and \eqref{df13d}.

{\it Step 2: consider the case that lower subsystem of \eqref{ge} has an invisible tangent point.}
Since $\gamma^+_1\ne0$ when $\varsigma^+=0$ and $\Delta(x)=(e^{\gamma^+_1\pi}-1)x+o(x)$,
we get $\varsigma=0$.
Conclusion {\bf (d)} is proved.
Conclusion {\bf (e)} can be proved similarly as conclusion {\bf (b)} and hence we omit it here.
\end{proof}

%%%%%%%%%%%%%%%%%%%%%%%%%%%%%%%%%%%%%%%%%%%%%%%%%%%%%%%%%%%%%%%%%

\section{Corollaries and discussions }
\setcounter{equation}{0}
\setcounter{lm}{0}
\setcounter{thm}{0}
\setcounter{rmk}{0}
\setcounter{df}{0}
\setcounter{cor}{0}

In this section, we give some interesting corollaries and discussions related to normal forms and Lyapunov constants.

An analytic system $(\dot{x},\dot{y})^\top={\boldsymbol X}(x,y)$ can be treated as a piecewise-smooth
system~\eqref{ge} with two same subsystems, so we can get the following corollary directly from Theorem~\ref{NFFFPP} and
omit its proof here.

\begin{cor}
An analytic system $(\dot{x},\dot{y})^\top={\boldsymbol X}(x,y)$ having a weak focus at the origin
with  eigenvalues $\alpha \pm \beta i$ can be
rewritten as system~\eqref{nffs} for any positive integer $N$ by a diffeomorphism~\eqref{dhfs} and a time scaling~\eqref{tsfs}.
  \label{smnf}
\end{cor}

According to Theorem~\ref{Lya-FFPP}, for \eqref{nffs} we get $V_1=e^{\alpha\pi}-e^{-\alpha\pi}$, $V_k=0$ for even $k$ and $V_k= 2\gamma_{k}C_{k}^{FF}$ for odd $k\ge 3$.
Thus we obtain that the index of first nonzero Lyapunov constant is always odd, which is consistent with previous result (see \cite{KUZ,ROM,ZZ})
for analytic system  $(\dot{x},\dot{y})^\top={\boldsymbol X}(x,y)$. That is, Corollary~\ref{smnf} provides a new method to compute Lyapunov constants for analytic systems by reducing to a new normal form.

It is worth mentioning that previous normal form for $(\dot{x},\dot{y})^\top={\boldsymbol X}(x,y)$ as
\begin{equation}\label{snff}
  \left(\begin{array}{c}
      \dot{x} \\
      \dot{y} \\
    \end{array}\right)
    =\left(\begin{array}{cc}
      \alpha & -\beta \\
      \beta & \alpha \\
    \end{array}\right)
  \left(\begin{array}{c}
      x \\
      y \\
    \end{array}\right)
    +\sum_{k=1}^{N}(x^2+y^2)^k
    \left(\begin{array}{cc}
      \gamma_{2k+1} & -\eta_{2k+1} \\
      \eta_{2k+1} & \gamma_{2k+1} \\
    \end{array}\right)\left(\begin{array}{c}
      x \\
      y \\
    \end{array}\right)
    +o\left(|(x,y)|^{2N+1}\right)
\end{equation}
is obtained by a linear transformation and a near-identity (see \cite{GUC,Li}).
Moreover, for Lyapunov constants $V_k$ we have $V_1=e^{2\alpha \pi}-1$, $V_k=0$ for even $k$ and $V_k$ are algebraically equivalent to $\gamma_{2k+1}$ for odd $k$ and $V_1=\cdots=V_{k-1}=0$.
Note that time scaling is not required in reducing to normal form~\eqref{snff}.
However, we usually get normal form~\eqref{snff} with the help of corresponding complex system,
and the conversion between real and complex systems may cause more computations.
It is not hard to see that normal form~\eqref{nffs} can be obtained directly in the real coordinate by Lemma~\ref{niofs}.

%1.截断系统中心

\begin{cor}
The origin $O$ is a center of truncation of FF normal form~\eqref{NFF}
 \begin{align}\label{FFN}
		\left( \begin{array}{c}
		\dot{x}\\
		\dot{y}\\
	\end{array} \right) =\left\{
    \begin{aligned}
&		\left(\begin{array}{cc}
       \gamma_1^+ & -1 \\
       1 & \gamma_1^+ \\
     \end{array}\right)
     \left(\begin{array}{c}
         x \\
         y \\
       \end{array}\right)+
       \sum_{i=1}^{N}\gamma_{i+1}^+y^k\left(\begin{array}{c}
         x \\
         y \\
       \end{array}\right)~~~~&&\mathrm{if}~y>0,\\
&		\left(\begin{array}{cc}
       \gamma_1^- & -1 \\
       1 & \gamma_1^- \\
     \end{array}\right)
     \left(\begin{array}{c}
         x \\
         y \\
       \end{array}\right)+
       \sum_{i=1}^{N}\gamma_{i+1}^-y^k\left(\begin{array}{c}
         x \\
         y \\
       \end{array}\right)   && \mathrm{if}~y<0
	\end{aligned} \right.
\end{align}
if and only if $\gamma_{k}^+-(-1)^k\gamma_{k}^-=0, k=1,...,N+1$,
 a center of truncation of FP normal form~\eqref{NFP}
\begin{align}\label{FPN}
		\left( \begin{array}{c}
		\dot{x}\\
		\dot{y}\\
	\end{array} \right) =\left\{
    \begin{aligned}
&		\left(\begin{array}{cc}
       \gamma_1^+ & -1 \\
       1 & \gamma_1^+ \\
     \end{array}\right)
     \left(\begin{array}{c}
         x \\
         y \\
       \end{array}\right)+
       \sum_{i=1}^{N}\gamma_{i+1}^+y^i\left(\begin{array}{c}
         x \\
         y \\
       \end{array}\right)~~~~&&\mathrm{if}~y>0,\\
&	   \left(\begin{array}{c}
         1 \\
         x^{2\ell^--1}\\
       \end{array}\right)  && \mathrm{if}~y<0
	\end{aligned} \right.
\end{align}
if and only if $\gamma_k^+=0, k=1,...,N+1$,
 a center of truncation of PP normal form~\eqref{NPP}
\begin{align}\label{PPN}
		\left( \begin{array}{c}
		\dot{x}\\
		\dot{y}\\
	\end{array} \right) =\left\{
    \begin{aligned}
&		\left(\begin{array}{c}
         -1 \\
         x^{2\ell^+-1}+\sum\limits_{i=1}^{N}\sigma_{i+1}^+x^{2\ell^+-1+i}\\
       \end{array}\right)
        && \mathrm{if}~y>0,\\
&		      \left(\begin{array}{c}
         1 \\
         x^{2\ell^--1}\\
       \end{array}\right) && \mathrm{if}~y<0
	\end{aligned} \right.
\end{align}
if and only if $\sigma^+_{2k}=0, k=1,...,\lfloor(N+1)/2\rfloor$.
\end{cor}
\begin{proof}
It is not hard to show that the necessity can be obtained directly by Theorem~\ref{Lya-FFPP}.
For the sufficiency, $O$ is a center of system~\eqref{FFN} because of the symmetry under $(x,y,t)\to(x,-y,-t)$ when $\gamma_{k}^+-(-1)^k\gamma_{k}^-=0, k=1,...,N+1$.
similarly, we can get that system~\eqref{FPN} with $\gamma_k^+=0, k=1,...,N+1$ and system~\eqref{PPN} with $\sigma^+_{2k}=0, k=1,...,\lfloor(N+1)/2\rfloor$ are invariant under $(x,y,t)\to(-x,y,-t)$.
This lemma is proved.
\end{proof}

%2.pp第一阶非退化Lyapunov量
\begin{cor}
  The order of monodromic singular point $O$ of PP type of system~\eqref{ge} is always fraction, i.e.,
  index of the first non-vanishing Lyapunov constant is always even.
  \label{FNL}
\end{cor}

Corollary~\ref{FNL} can be obtained directly from Theorems~\ref{NFFFPP} and \ref{Lya-FFPP}, and here we omit its proof.
For (2,2)-monodromic tangential singularity, i.e., a special case of PP type,
conclusion of Corollary~\ref{FNL} is given in \cite[p.235]{FAF} by properties of involutions for the half-return maps
and in \cite{EM1} by normal form method.
For general case of PP type, Corollary~\ref{FNL} is obtained in \cite{DD}  as a consequence of property for pair of involution.
But in this paper we get Corollary~\ref{FNL} by a different method, i.e., providing a new normal form method to prove that index of first non-vanishing Lyapunov constant is always even for general case of PP type.

%3.FF中心→k+=k-;FP中心→F中心.

\begin{cor}
If the origin $O$ of system~\eqref{ge} is a center of FF type, then $O$ has the same orders for two subsystems.
If the origin $O$ of system~\eqref{ge} is a center of FP type, then $O$ is a center of the focus-side subsystem.
\label{offfp}
\end{cor}

Corollary~\ref{offfp} can be obtained directly by Theorem~\ref{OFF} and here we omit its proof. Corollary~\ref{offfp}
provides necessary conditions for $O$ to be a center of FF type or FP type, which usually help us find center conditions.

In \cite{CX}, system~\eqref{ge} with a monodromic singular point $O$ of FF type is reduced to
\begin{align}\label{cxnf}
  \left(
    \begin{array}{c}
      \dot{x} \\
      \dot{y} \\
    \end{array}
  \right)=\left\{\begin{aligned}
  &\left(
    \begin{array}{cc}
      \alpha^+ & \beta^+ \\
      -\beta^+ & \alpha^+ \\
    \end{array}
  \right)\left(
  \begin{array}{c}
             x \\
             y \\
  \end{array}\right)+\left(
  \begin{array}{c}
      \gamma^+(x^2+y^2) \\
      \eta^+(x^2+y^2) \\
    \end{array}
  \right)+o(|(x,y)|^2)~~~~~&&\textrm{if}~y>0,\\
  &\left(
    \begin{array}{cc}
      \alpha^- & \beta^- \\
      -\beta^- & \alpha^- \\
    \end{array}
  \right)\left(
  \begin{array}{c}
             x \\
             y \\
  \end{array}\right)+\left(
  \begin{array}{c}
      \gamma^-(x^2+y^2) \\
      0 \\
    \end{array}
  \right)+o(|(x,y)|^2)~~~~~&&\textrm{if}~y<0
  \end{aligned}\right.
\end{align}
by homeomorphism~\eqref{geni}.
Additionally, for normal form \eqref{cxnf} the first Lyapunov constant is $V_1=e^{-\pi\alpha^+/\beta^+}-e^{\pi\alpha^-/\beta^-}$ and the second one $V_2$ is algebraically equivalent to $-\eta^+/\beta^+$ when $V_1=0$.
For the convenience of calculating Lyapunov constants, it is better to choose a normal form satisfying that
$k$-th Lyapunov constant $V_k$ depends only on the coefficients of $k$-th terms as in classic normal form theory when $V_1=\cdots=V_{k-1}=0$.
However, \eqref{cxnf} does not have such property because it is not hard to obtain that $\gamma^{\pm}$ in second term of
\eqref{cxnf} appear in fourth Lyapunov constant $V_4$ when $V_1=V_2=V_3=0$ and can not be further eliminated by homeomorphism~\eqref{geni}.
In this paper, we reduce second order term of system~\eqref{ge} to $\gamma_2^{\pm}y(x,y)^\top$ for $\pm y>0$ to replace the
second order term of \eqref{cxnf}.
Clearly, parameters $\gamma_2^{\pm}$ in new normal form \eqref{NFF} can only influence the second Lyapunov constant $V_2$ by Theorem~\ref{Lya-FFPP},
not $V_k$ ($k\ge 3$). Moreover, following \eqref{NFF} we can get laconic expression of each $V_k$ ($k\ge 1$) in Theorem~\ref{Lya-FFPP}.
Therefore, for FF type we prefer the new normal form \eqref{NFF} given in Theorem~\ref{NFFFPP}.

From \cite{EM}, system~\eqref{ge} with a monodromic singular point $O$ of PP type can be reduced to
\begin{align}
		\left( \begin{array}{c}
		\dot{x}\\
		\dot{y}\\
	\end{array} \right) =
		\left(\begin{array}{c}
         \mp 1\\
         x^{2\ell^{\pm}-1}+\sum\limits_{i=1}^{N}\sigma_{i+1}^{\pm}
         x^{2\ell^{\pm}-1+i}\\
       \end{array}\right)+
      {\boldsymbol R}^{\pm}_{N}(x,y)~~~~~~~\mathrm{if}~\pm y>0
\label{EMPP}
\end{align}
for any positive $N$, where $\sigma_{i+1}^{\pm}=0$ if $2\ell^{\pm}\mid i$.
Indeed, system~\eqref{EMPP} is obtained from the normal form given in \cite[Theorem 1]{EM} by transformation $(x,y,t)\to(y,x,-t)$ and letting $\sigma_{i+1}^{\pm}:=-\mu^{\pm}_{i+2\ell^{\pm}-1}$.
By \cite[Proposition 6]{EM},
the Lyapunov constants for system~\eqref{EMPP} are $V_1=0$ and
\begin{equation}
V_k=\left\{
  \begin{aligned}
  &\frac{2\sigma_2^+}{2+2\ell^+}-\frac{2\sigma_2^-}{2+2\ell^-}
  ~~~~&&{\rm if}~k=2,\\
  &(1\!+\!(-1)^k)\left(\frac{\sigma_k^+}{k+2\ell^+-1}\!-\!
  \frac{\sigma_k^-}{k+2\ell^--1}\right)\!+\!\widetilde\Psi(\sigma_2^\pm,...,\sigma_{k-1}^\pm)&&{\rm if}~k\ge 3,
  \end{aligned}\right.
  \label{vkex}
\end{equation}
but the expression of $\widetilde\Psi(\sigma_2^\pm,...,\sigma_{k-1}^\pm)$ is not given.
According to Lemma~\ref{lcops} of this paper, $V_k$ takes form \eqref{vkex} and $\widetilde\Psi(\sigma_2^\pm,...,\sigma_{k-1}^\pm)=
\Psi(\sigma_2^+,...,\sigma_{k-1}^+)
  -\Psi(\sigma_2^-,...,\sigma_{k-1}^-)$, where $\Psi$ is explicitly given below \eqref{rtmap}.
Since $\Psi(\sigma_2^+,...,\sigma_{k-1}^+)
  -\Psi(\sigma_2^-,...,\sigma_{k-1}^-)$ is complicated and it is not hard to check that it does not vanish
when $V_1=\cdots=V_{k-1}=0$, for system~\eqref{EMPP} it is difficult to judge $V_k$ vanish or not when $V_1=\cdots=V_{k-1}=0$.
Therefore, we are motivated to provide a PP normal form which has more simpler form of Lyapunov constants.
Note that in our normal form~\eqref{NPP}, the terms $x^{2\ell^--1+i}$ have been further eliminated.
Thus, although $\sigma^+_{i+1}$ of \eqref{NPP} may not be zero when $2\ell^{\pm}\mid i$,
the first nonzero Lyapunov constant $V_k$ can be directly obtained and only depend on $\sigma^+_k$ as shown in Theorem~\ref{Lya-FFPP}.
In the following we explain the reason why we can get a simpler normal form \eqref{NPP} than normal form \eqref{EMPP}.
The key to eliminate more terms in the lower subsystem is that the homeomorphism~\eqref{geni} we used in this paper
allows points on the switching line to change, but homeomorphism used in \cite{EM1, EM} does not.
Note that such change does not influence homeomorphism as proved in Lemma~\ref{niih}.
More precisely, from step 3 of the proof of Theorem~\ref{NFFFPP} we get that the lower half map of homeomorphism~\eqref{geni} changes the lower subsystem of \eqref{ge} into system~\eqref{npN} and transforms point $(x,0)$ to $\left(x+\sum_{k=2}^{N+1}r^-_{k,0}x^{k},0\right)$.
Thanks to the special form of system~\eqref{npN}, it can be further reduced to system~\eqref{nfps} by time scaling~\eqref{tsps}.
That is, terms $x^{2\ell^--1+i}$ in the lower subsystem of \eqref{EMPP} have been eliminated by our approach.
Clearly, the upper half map of homeomorphism~\eqref{geni} must make same change for the points on the switching line by Lemma~\ref{niih}.
Thus the upper subsystem of \eqref{ge} can only be reduced to system~\eqref{npNns} by homeomorphism~\eqref{geni},
and then from time scaling~\eqref{tsfpp} we get the upper subsystem of PP normal form~\eqref{NPP}.

%%%%%%%%%%%%%%%%%%%%%%%%%%%%%%%%%%%%%%%%%%%%%%%%%%%%%%%%%%%%%%%%%

\section{Applications}

\setcounter{equation}{0}
\setcounter{lm}{0}
\setcounter{thm}{0}
\setcounter{rmk}{0}
\setcounter{df}{0}
\setcounter{cor}{0}

In this section, we apply our main results to center problem or to several systems investigated in previous publications to correct some clerical errors.

\begin{exam}
  Consider the center problem of the following piecewise-smooth quadratic system
\begin{align}\label{exqs}
  \left(
    \begin{array}{c}
      \dot{x} \\
      \dot{y} \\
    \end{array}
  \right)=\left\{\begin{aligned}
  &\left(
  \begin{array}{c}
      -y+p_{20}x^2+p_{11}xy+p_{02}y^2 \\
      x+q_{20}x^2+q_{11}xy+q_{02}y^2 \\
    \end{array}
  \right)~~~~~&&{\rm if}~y>0,\\
  &\left(
  \begin{array}{c}
             -y \\
              x \\
  \end{array}\right)~~~~~&&{\rm if}~y<0.
  \end{aligned}\right.
\end{align}
\end{exam}
Clearly, the origin $O$ of system~\eqref{exqs} is a monodromic singular point of FF type.
It is stated in \cite[Theorem 3.1]{GAS} that the center variety of system~\eqref{exqs} is decomposed into five components, and the fifth one
is
$${\rm (v):}~~ 2p_{11}q_{20}+3p_{20}^2-2q_{20}^2=2q_{11}+5p_{20}=8p_{02}q_{20}^2-3p_{20}^2+8q_{20}^2=4q_{20}^2+4q_{02}q_{20}-3p_{20}^2=0.$$
However, it is not hard to check that $O$ is not a center by simulations when $p_{20}=q_{11}=0$ and $p_{11}=-p_{02}=q_{20}=-q_{02}=1$, which satisfies (v). Actually,
by our Theorems~\ref{NFFFPP}, \ref{Lya-FFPP}
% and the programme given in Appendix 
we find that the fifth center condition (v) has
clerical errors and should be corrected as
\begin{equation}
  p_{11}+2q_{02}+q_{20}=2q_{11}+5p_{20}=2p_{02}q_{20}-p_{20}q_{02}
  +p_{20}q_{20}=4q_{20}^2+4q_{02}q_{20}-3p_{20}^2=0.
  \label{ccqs5}
\end{equation}
In fact, 
%by the programme listed in Appendix 
we get $V_1\equiv 0$ and
\begin{align*}
  V_2&=\frac{2}{3}(p_{11}+2q_{02}+q_{20}),
  &&V_3=\!-\!\frac{\pi}{8}
  (p_{02}q_{20}\!+\!2p_{20}q_{02}\!+\!3p_{20}q_{20}\!+\!q_{11}q_{02}\!+\!q_{11}q_{20}),\\
  V_4&=\frac{2}{15}
  q_{20}(6p_{02}p_{20}-p_{11}^{2}+3p_{20}^{2}+q_{20}^{2}),
  &&V_5=\frac{\pi}{64}
  p_{20}q_{20}(p_{20}^{2}-2q_{11}p_{20}-8q_{02}q_{20}-8q_{20}^{2}),\\
  V_6&=-\frac{8}{315}q_{20}q_{02}^{2}(4q_{20}^{2}+4q_{02}q_{20}-3p_{20}^{2}).
\end{align*}
Here $V_{i+1}$ is expressed as above under $V_1=\cdots=V_i=0$.
Solving $\{V_2=\cdots=V_6=0\}$, we obtain the first four components as given in \cite[Theorem 3.1]{GAS} and the fifth
one stated in \eqref{ccqs5}.
It is not hard to verify that system~\eqref{exqs} with condition \eqref{ccqs5} has a first
integral given by
\begin{equation*}
  H(x,y)=\left\{\begin{aligned}
  &(2q_{20}x\!-\!p_{20}y\!-\!2)^{2}\left((3 p_{20}^{2}\!-\!8q_{20}^{2}) y^{2}\!-\!12p_{20}q_{20}xy\!-\!4p_{20}y\!+\!(2q_{20}x\!+\!2)^{2}\right)~~~&&{\rm if}~y\ge0,\\
  &4q_{20}^2(x^2+y^2)-4&&{\rm if}~y\le0,
  \end{aligned}\right.
\end{equation*}
and hence $O$ is a center by \cite[Proposition 2.1]{CX1} since $H(-x,0)=H(x,0)$.
%Note that $p_{20}=-2,p_{11}=-\sqrt{3},p_{02}=1,q_{20}=\sqrt{3},q_{11}=5,q_{02}=0$ satisfies
%\eqref{ccqs5}, but does not satisfy any conditions in \cite[Theorem 3.1]{GAS}.

\begin{exam}
  Consider the center problem of the piecewise-smooth system
  \begin{align}\label{exsbs}
  \left(
    \begin{array}{c}
      \dot{x} \\
      \dot{y} \\
    \end{array}
  \right)=\left\{\begin{aligned}
  &\left(
  \begin{array}{c}
      \delta x-y-a_3x^2+(2a_2+a_5)xy+a_6y^2 \\
      x+\delta y+a_2x^2+(2a_3+a_4)xy-a_2y^2 \\
    \end{array}
  \right)~~~~~&&{\rm if}~y>0,\\
  &\left(
  \begin{array}{c}
      1+(b_1+b_3)x+b_2y \\
      x+(b_1-b_3)y \\
  \end{array}\right)~~~~~&&{\rm if}~y<0.
  \end{aligned}\right.
\end{align}
\end{exam}

Clearly, the origin $O$ is a monodromic singular point of FP type.
According to Corollary~\ref{offfp}, to obtain center conditions of system~\eqref{exsbs} we only need to
compute Lyapunov constants under the four classes of center conditions of the upper subsystem which are obtained in \cite{Bau}
as
\begin{eqnarray*}
\begin{aligned}
&{\rm (I)}: ~~\delta=a_4=a_5=0,~~~~&&{\rm (II)}: ~~\delta=a_3-a_6=0,\\
&{\rm (III)}: ~~\delta=a_5=a_4+5a_3-5a_6=a_3a_6-2a_6^2-a_2^2=0, &&{\rm (IV)}: ~~\delta=a_2=a_5=0.
\end{aligned}
\end{eqnarray*}
In the following, Lyapunov constant $V_k$ can be obtained by our Theorems~\ref{NFFFPP}, \ref{Lya-FFPP} and
simplified by $V_1=\cdots=V_{k-1}=0$.
For the case (I),
we obtain $V_1=V_3\equiv 0$ and
\begin{equation*}
  V_2=\frac{2 a_{2}}{3}+\frac{4 b_{1}}{3},~~~~~~~~~~~
  V_4=\frac{a_{2} (9 a_{2}^{2}-4 b_{3}^{2}-4 b_{2})}{30}.
\end{equation*}
Solving $\{V_1=...=V_4=0\}$, we can obtain some necessary center conditions. Similarly, for the cases (II), (III), (IV) we
also obtain some necessary center conditions respectively. Finally, summarizing all the these necessary conditions we get
that if the origin $O$ is a center of system~\eqref{exsbs}, then  one of the following eight conditions holds:
\begin{eqnarray}
\begin{aligned}
&{\rm (C1)}: ~\delta=a_4=a_5=a_2+2b_1=9b_1^2-b_3^2-b_2=0, &&
\\
&{\rm (C2)}: ~\delta=a_3=a_6=a_2+a_5=b_1=0, &&\\
&{\rm (C3)}: ~\delta=a_2=a_5=b_1=0, &&
\\
&{\rm (C4)}: ~\delta=a_2=a_3-a_6=a_5+2b_1=a_6a_4+3a_6^2+b_1^2-b_3^2-b_2=0,&&\\
&{\rm (C5)}: ~\delta=a_3=a_6=a_2+a_5+2b_1=a_2^2-2a_2b_1+b_1^2-b_3^2-b_2=0, &&\\
&{\rm (C6)}: ~\delta=a_3=a_6=4a_2+a_5=2b_1-3a_2=b_1^2-9b_3^2-9b_2=0, &&\\
&{\rm (C7)}: ~\delta=a_3-a_6=4a_2+a_5=2b_1-3a_2=a_4+4a_6=b_1^2-9b_3^2-9b_2=0,&&\\
&{\rm (C8)}: ~\delta=a_3=a_5=a_6=a_2+2b_1=9b_1^2-b_3^2-b_2=0.&&
\end{aligned}
\label{eightC}
\end{eqnarray}
We claim that conditions (C1), ..., (C8) in \eqref{eightC} are sufficient for $O$ to be a center of system~\eqref{exsbs}.
In fact, the first integral $H^+(x,y)$ of the upper subsystem under \eqref{eightC} can be found in \cite{CX0,TIAN}.
It is not hard to compute that the first integral of the lower subsystem can be given by $H^-(x,y)=1-\omega x^{2}+2 \omega (x b_{3}+1) y+b_{2} \omega y^{2}$
for $b_1=0$ and
\begin{equation*}
  H^-(x,y)\!=\!\left\{\begin{aligned}
  &\!(1\!+\!\alpha^+x\!+\!\alpha^+(\sqrt{\omega}\!-\!b_{3})y)^{\frac{\alpha^- }{\sqrt{\omega}}}(1\!-\!\alpha^-x\!+\!\alpha^-(\sqrt{\omega}\!+\!b_{3})y)^
  {\frac{\alpha^+}{\sqrt{\omega}}}&&{\rm if}~\omega>0~{\rm and}~ \gamma\ne0,\\
  &\!(1\!+\!2b_1x\!+\!2b_1(b_1\!-\!b_3)y)\exp(2b_1((b_1\!+\!b_3)y\!-\!x))&&{\rm if}~\omega>0~{\rm and}~\gamma=0,\\
  &\!(1\!+\!b_1x\!-\!b_1b_3y)\exp\left(\frac{-x b_{1}^{2}\!-\!b_{1}\!+\!b_{3}}{(1\!+\!b_1x\!-\!b_1b_3y) b_{3}}\right)&&{\rm if}~\omega=0~{\rm and}~b_3\ne0,\\
  &\!(1\!+\!b_{1}x)\exp\left(\frac{1\!-\!b_{1}^{2} y}{1\!+\!b_{1}x}\right)&&{\rm if}~\omega=b_3=0,\\
  &\!(1\!+\!2b_1x\!+\!\gamma x^2\!-\!2(b_3\gamma x\!-\!b_1b_3\!+\!\omega)y\!-\!b_2\gamma y^2)\exp\left(\varphi(x,y)\right)&&{\rm if}~\omega<0
  \end{aligned}\right.
\end{equation*}
for $b_1\ne0$,
where $\omega:=b_{3}^{2}+b_{2}, \alpha^{\pm}:=\sqrt{\omega}\pm b_1, \gamma:=b_1^2-\omega, \eta:=(b_3-b_1)/\gamma$ and
\begin{equation*}
  \varphi(x,y):=\left\{\begin{aligned}
  &\frac{2b_1}{\sqrt{-\omega}}\left(\arctan\left(\frac{b_3}{\sqrt{-\omega}}\!+\!
  \frac{b_2(\gamma y\!-\!1)}{\sqrt{-\omega}
  (\gamma x\!+\!b_1\!-\!b_3)}\right)+\frac{{\rm sgn}(\eta\!-\!x)\pi}{2} \right)~~~~&&{\rm for}~x\ne\eta,\\
  &0&&{\rm for}~x=\eta.
  \end{aligned}\right.
\end{equation*}
For any conditions in \eqref{eightC}, we obtain that either there exists constants $K$ and $C$ such that $H^+(x,0)\equiv K H^-(x,0)+C$ , or both $H^{\pm}(x,0)$ are even functions, or $H^{\pm}(x,0)=H^{\pm}(\rho,0)$ have common
zeros $x(\rho)$ satisfying $x(\rho)\to 0^-$ as $\rho\to 0^+$.
Therefore, (C1), ..., (C8) in \eqref{eightC} are necessary
and sufficient conditions for $O$ to be a center of system~\eqref{exsbs}.

\begin{exam}
  Let $k^{\pm}$ be any positive integers and consider system
  \begin{align}\label{exkm}
  \left(
    \begin{array}{c}
      \dot{x} \\
      \dot{y} \\
    \end{array}
  \right)=\left\{\begin{aligned}
  &\left(
  \begin{array}{c}
      1 \\
      -x^{2k^+-1}(\lambda x+1) \\
    \end{array}
  \right)~~~~~&&{\rm if}~y>0,\\
  &\left(
  \begin{array}{c}
             -1 \\
              x^{2k^--1}(x-1)\\
  \end{array}\right)~~~~~&&{\rm if}~y<0.
  \end{aligned}\right.
\end{align}
\end{exam}

For system~\eqref{exkm}, the origin $O$ is a monodromic singular point of PP type and the Hopf bifurcations of \eqref{exkm} are investigated in \cite[Example 1]{DD}.
However, the Lyapunov constant $V_2$ has a clerical error in \cite[Example 1]{DD} and then
$V_4$ is also wrong. This leads to a wrong bifurcation value of $\lambda$ in the Hopf bifurcations.
In fact, from our normal form method it is not hard to obtain that the first four Lyapunov constants are given by $V_1=V_3\equiv 0$ and
\begin{equation*}
  V_2=-\frac{2\lambda}{1+2k^+}-\frac{2}{1+2k^-},~~~
  V_4=\frac{4(k^+-k^-)(2k^++2k^-+3)}{3(1+2k^-)^3}.
\end{equation*}
Solving $V_2=V_4=0$ we obtain $k^+-k^-=\lambda+1=0$, under which $O$ is a center of \eqref{exkm} by
the symmetry with respect to $x$-axis.
For the Hopf bifurcations, consider $k^+\ne k^-$, implying $V_4\ne 0$.
In the case of $k^+>k^-$ (resp. $k^+<k^-$), for $\lambda_0=-(1+2k^+)/(1+2k^-)$ we get $V_2(\lambda_0)=0$, $V'_2(\lambda_0)\ne0$ and $V_4(\lambda_0)>0$ (resp. $<0$).
Thus, from \cite[Theorem D]{DD} we obtain that \eqref{exkm} exists one unstable (resp. stable) limit cycle when
$\lambda$ changes great (resp. less) from $\lambda_0$.

\section{Acknowledgements}

The authors thank the support of National Key R\&D Program of China (No. 2022YFA1005900), National Natural Science Foundation
of China (No. 12271378, No. 12171336),
Sichuan Science and Technology Program (No. 2024NSFJQ0008).


{\footnotesize
\begin{thebibliography}{99}

\bibitem{AL}
A. Algaba, E. Freire, E. Gamero, C. Garcia,
Quasi-homogeneous normal forms,
{\it J. Comput. Appl. Math.} {\bf 150}(2003), 193-216.

\bibitem{AN}
D. V. Anosov, V. I. Arnold,
{\it Dynamical Systems I: Ordinary Differential Equations and Smooth Dynamical Systems},
Encyclopedia of Mathematical Science, Vol. 1,
Springer, Berlin, 1988.

\bibitem{AR}
V. Arnold,
{\it Geometrical Methods in the Theory of Ordinary Differential Equations},
Springer, Berlin, 1988.

\bibitem{Bau}
N. N. Bautin,
On the number of limit cycles which appear with the variation of coefficients from an equilibrium position of
focus or center type,
{\it Trans. Amer. Math. Soc.} {\bf 100}(1954), 1-19.

\bibitem{BEBU}
M. di Bernardo, C. J. Budd, A. R. Champneys,
Grazing, skipping and sliding: analysis of the nonsmooth dynamics of the DC/DC buck converter,
{\it Nonlinearity} {\bf 11}(1998), 859-890.

\bibitem{BDM}
M. di Bernardo, C. J. Budd, A. R. Champneys, P. Kowalczyk,
{\it Piecewise-Smooth Dynamical Systems: Theory and Applications}, Springer, London, 2008.

\bibitem{BUZC}
C. A. Buzzi, T. D. Carvalho, M. A. Teixeira,
Birth of limit cycles bifurcating from a
nonsmooth center,
{\it J. Math. Pures Appl.} {\bf 102}(2014), 36-47.


\bibitem{BUZ}
C. A. Buzzi, J. R. Medrado, M. A. Teixeira,
Generic bifurcation of refracted systems,
{\it Adv. Math.} {\bf 234}(2013), 653-666.


\bibitem{CH}
H. Chen, S. Duan, Y. Tang, J. Xie,
Global dynamics of a mechanical system with dry friction,
{\it J. Differential Equ.} {\bf 265}(2018), 5490-5519.

\bibitem{CHE}
H. Chen, J. Llibre, Y. Tang,
Global dynamics of a SD oscillator,
{\it Nonlinear Dyn.} {\bf 91}(2018), 1755-1777.

\bibitem{CX0}
X. Chen, Z. Du,
Limit cycles bifurcate from centers of discontinuous quadratic systems,
{\it Comput. Math. Appl.} {\bf 59}(2010), 3836-3848.

\bibitem{CX1}
X. Chen, W. Zhang,
Isochronicity of centers in a switching Bautin system,
{\it J. Differential Equ.} {\bf 252}(2012), 2877-2899.


\bibitem{CX2}
X. Chen, V. G. Romanovski, W. Zhang,
Degenerate Hopf bifurcations in a family of FF-type switching
systems,
{\it J. Math. Anal. Appl.} {\bf 432}(2015), 1058-1076.

\bibitem{CX}
X. Chen, W. Zhang,
Normal  forms  of  planar  switching  systems,
{\it Discrete Contin. Dyn. Syst.} {\bf 36}(2016), 6715-6736.

\bibitem{CHO}
S. -N. Chow, C. Li, D. Wang,
{\it Normal Forms and Bifurcation of Planar Vector Fields},
Cambridge University Press, 1994.

\bibitem{COM}
L. Comtet,
{\it Advanced Combinatorics: The Art of Finite and Infinite Expansions, enlarged edition},
D. Reidel Publishing Co., Dordrecht, 1974.


\bibitem{COO}
S. Coombes, R. Thul, K. C. A. Wedgwood,
 Nonsmooth dynamics in spiking neuron models,
{\it Phys. D} {\bf 241}(2012), 2042-2057.

\bibitem{CO}
B. Coll, A. Gasull, R. Prohens,
Degenerate Hopf Bifurcations in Discontinuous Planar Systems,
{\it J. Math. Anal. Appl.} {\bf 253}(2001), 671-690.

\bibitem{DUM}
F. Dumortier, J. Llibre, J. C. Art\'es,
{\it Qualitative Theory of Planar Differential Systems}, Springer Verlag, Berlin, 2006.

\bibitem{EM1}
M. Esteban, E. Freire, E. Ponce, F. Torres,
On normal forms and return maps for pseudo-focus points,
{\it J. Math. Anal. Appl.} {\bf 507}(2022), 125774.

\bibitem{EM}
M. Esteban, E. Freire, E. Ponce, F. Torres,
Piecewise smooth systems with a pseudo-focus: a normal form approach,
{\it Appl. Math. Model.} {\bf 115}(2023), 886-897.

\bibitem{FA}
T. Faria, L.T. Magalhaes,
Normal forms for retarded functional differential equations
with parameters and applications to Hopf bifurcation,
{\it J. Differential Equ.} {\bf 122}(1995), 181-200.


\bibitem{FAR}
T. Faria, L.T. Magalhaes,
Normal forms for retarded functional differential equations and
applications to Bogdanov-Takens singularity,
{\it J. Differential Equ.} {\bf 122}(1995), 201-224.

\bibitem{FAF}
A. F. Filippov,
{\it Differential Equation with Discontinuous Righthand Sides},
Kluwer Academic Publishers, Dordrecht, 1988.

\bibitem{GAS}
A. Gasull, J. Torregrosa,
Center-Focus problem for discontinuous planar differential equations,
{\it Int. J. Bifurc. Chaos}, {\bf 13}(2003), 1755-1765.

\bibitem{GUC}
J. Guckenheimer, P. Holmes,
{\it Nonlinear Oscillations, Dynamical Systems, and Bifurcations of Vector Fields},
Springer Verlag, New York, 1983

\bibitem{GUO}
J. Guo,
Theory and applications of equivariant normal forms
and Hopf bifurcation for semilinear FDEs in Banach
spaces,
{\it J. Differential Equ.} {\bf 317}(2022), 387-421.

\bibitem{GUOZ}
Z. Guo, Y. Tang, W. Zhang,
More degeneracy but fewer bifurcations in a predator-prey system having fully null linear part,
{\it Z. Angew. Math. Phys.} {\bf 73}(2022), 122.

\bibitem{HAN}
M. Han, W. Zhang,
On Hopf bifurcation in non-smooth planar systems,
{\it J. Differential Equ.} {\bf 248}(2010), 2399-2416.

\bibitem{JEF}
M. R. Jeffrey, S. J. Hogan,
The geometry of generic sliding bifurcations,
{\it SIAM Rev.} {\bf 53}(2011), 505-525.

\bibitem{JMR}
M. R. Jeffrey,
{\it Hidden Dynamics: The Mathematics of Switches, Decisions and Other Discontinuous Behaviour},
Springer, New York, 2018.

\bibitem{KM}
M. Kunze,
{\it Non-smooth Dynamical Systems}, Springer, Berlin, 2000.

\bibitem{KUZ1}
Y. A. Kuznetsov, S. Rinaldi, A. Gragnani,
One-parameter bifurcations in planar Filippov systems,
{\it Int. J. Bifurc. Chaos} {\bf 13}(2003), 2157-2188.

\bibitem{KUZ}
Y. A. Kuznetsov,
{\it Elements of Applied Bifurcation Theory},
Springer, New York, 2004.

\bibitem{Li}
W. Li,
{\it Normal Form Theory and Applications}, Science Press, Beijing, 2000.

\bibitem{LIW}
W. Li, K. Lu,
Poincar\'e theorems for random dynamical systems,
{\it Erg. Theo. Dyn. Syst.} {\bf 25}(2005), 1221-1236.

\bibitem{LU}
K. Lu, W. Zhang, W. Zhang,
$C^1$ Hartman Theorem for random dynamical systems,
{\it Adv. Math.} {\bf 375}(2020), 107375.

\bibitem{DD}
D. D. Novaes, L. A. Silva,
Lyapunov coefficients for monodromic tangential singularities in Filippov vector fields,
{\it J. Differential Equ.} {\bf 300}(2021), 565-596.

\bibitem{ROM}
V. Romanovski, D. Shafer,
{\it The Center and Cyclicity Problems: A Computational Algebra Approach},
Birkh{\"a}user, Boston, 2009.

\bibitem{RUA}
S. Ruan, Y. Tang, W. Zhang,
Versal unfoldings of predator-prey systems with ratio-dependent functional response,
{\it J. Differential Equ.} {\bf249}(2010), 1410-1435.

\bibitem{SCH}
A. Schild, J. Lunze, J. Krupar, W. Schwarz,
Design of generalized hysteresis controllers for DC-DC switching power converters,
{\it IEEE Trans. Power Electr.} {\bf 24}(2009), 138-146.

\bibitem{TIAN}
Y. Tian, P. Yu,
Center conditions in a switching Bautin system,
{\it J. Differential Equ.} {\bf 259}(2015), 1203-1226.

\bibitem{TA}
A. Tonnelier, W. Gerstner,
Piecewise linear differential equations and integrate-and-fire neurons: insights from two dimensional membrane models,
{\it Phys. Rev. E} {\bf 67}(2003), 021908.

\bibitem{WE}
L. Wei, X. Zhang,
Normal form and limit cycle bifurcation of piecewise smooth differential systems with a center,
{\it J. Differential Equ.} {\bf 261}(2016), 1399-1428.

\bibitem{ZZ}
Z. Zhang, T. Ding, W. Huang, Z. Dong,
{\it Qualitative Theory of Differential Equations(Translations of Mathematical Monographs)},
American Mathematical Society, Providence, 1992.





\end{thebibliography}}
\end{document}